\documentclass[11pt]{article}
\usepackage{amsfonts,amsmath,amsthm,amssymb,enumerate,color}
\usepackage[makeroom]{cancel}
\usepackage{authblk}
\usepackage{comment}
\usepackage[shortlabels]{enumitem}
\setlist[enumerate,1]{wide, labelindent=0pt,label={\upshape(\roman*)}}

\numberwithin{equation}{section}
\newcommand{\beq}{\begin{equation}}
	\newcommand{\eeq}{\end{equation}}
\newcommand{\bea}{\begin{eqnarray}}
	\newcommand{\eea}{\end{eqnarray}}
\newcommand{\beas}{\begin{eqnarray*}}
	\newcommand{\eeas}{\end{eqnarray*}}

\newtheorem{theorem}{Theorem}[section]

\newtheorem{definition}[theorem]{Definition}
\newtheorem{proposition}[theorem]{Proposition}
\newtheorem{prop}[theorem]{Proposition}
\newtheorem{corollary}[theorem]{Corollary}
\newtheorem{lemma}[theorem]{Lemma}
\newtheorem{remark}[theorem]{Remark}
\newtheorem{example}[theorem]{Example}
\newtheorem{examples}[theorem]{Examples}
\newtheorem{foo}[theorem]{Remarks}
\newtheorem{hypothesis}[theorem]{Hypothesis}

\newenvironment{Remark}{\begin{remark}\rm}{\end{remark}}

\title{Propagation of Chaos on Riemannian Manifolds}

\author[1]{Myeongju Chae\thanks{\texttt{mchae@hknu.ac.kr}}}
\author[2]{Gunhee Cho\thanks{\texttt{wvx17@txstate.edu}}}
\author[2]{Hyun Chul Jang\thanks{\texttt{hcjang@txstate.edu}}}
\author[3]{Guang Yang\thanks{\texttt{yangg7@sustech.edu.cn}}}

\affil[1]{Department of Applied Mathematics,
	Hankyong National University,
	Anseong, Republic of Korea}

\affil[2]{Department of Mathematics,
	Texas State University,
	San Marcos, TX, USA}

\affil[3]{Department of Mathematics,
	Southern University of Science and Technology,
	Shenzhen, China}

\date{}

\begin{document}
	
	\maketitle
	
\begin{abstract}
We extend reflection coupling for uniform-in-time propagation of chaos
from Euclidean space to complete Riemannian manifolds. Intrinsic
reflection along minimizing geodesics replaces the Euclidean difference
process, and the associated endpoint-index term quantifies the competition
between curvature, confinement, and interaction. Under an endpoint-index
bound and a radial Laplacian bound, we establish global well-posedness,
uniform moment estimates, and uniform-in-time propagation of chaos with
rate \(N^{-1/2}\) in a modified Wasserstein distance. We further obtain
exponential contraction of the nonlinear McKean--Vlasov semigroup,
uniqueness of its invariant probability measure in \(\mathcal P_2(M)\), and exponential
convergence to equilibrium. Although a Ricci lower bound is an important
special case, our framework also includes a curvature-spike class with
\(\inf_M\operatorname{Ric}=-\infty\).
\end{abstract}

	\tableofcontents

\section{Introduction}

McKean--Vlasov diffusions arise as mean-field limits of weakly interacting
particle systems as the number of particles tends to infinity. Propagation of
chaos describes the corresponding asymptotic independence of finitely many
particles. In Euclidean space, propagation of chaos for regular interactions,
such as globally Lipschitz forces, is classical; see
\cite{McKean1967,Sznitman1991,Meleard1996,Malrieu2003}. In particular,
the classical coupling argument gives quantitative estimates on any fixed
time interval. We refer to
\cite{ChaintronDiez2022,ChaintronDiezII} for detailed reviews of the
models, methods, and applications of propagation of chaos.

Uniform-in-time propagation of chaos requires estimates whose constants
remain bounded in time.
One approach is to  obtain the necessary dissipation from convexity assumptions on the
confinement and interaction potentials, using synchronous coupling or
functional inequalities such as logarithmic Sobolev inequalities; see
Malrieu~\cite{Malrieu2001} and Cattiaux, Guillin, and
Malrieu~\cite{CattiauxGuillinMalrieu2008}. A different approach is based
on reflection coupling. Reflection coupling was introduced by Lindvall and
Rogers~\cite{LindvallRogers1986} to construct successful couplings of
multidimensional diffusions. 
Eberle later employed this coupling to obtain quantitative Wasserstein contraction without global convexity~\cite{Eberle2016}. 
The approach was subsequently applied to
McKean--Vlasov processes and uniform-in-time propagation of chaos
\cite{EberleGuillinZimmer2019,DurmusEberleGuillinZimmer2020}.

In this paper, we extend the reflection-coupling approach to a complete
connected noncompact \(n\)-dimensional Riemannian manifold \((M,g)\). Let \(V\in C^2(M)\)
be a confinement potential and \(W\in C^2(M\times M)\) an interaction
potential. We consider the Vlasov--McKean equation
\begin{equation}\label{eq:intro-VM}
\partial_t\mu_t
=
\Delta_g\mu_t
+
\operatorname{div}_g\bigl((\nabla V+\mathbf F_{\mu_t})\mu_t\bigr),
\qquad
\mathbf F_\mu(x)
=
\int_M \nabla_g^{(1)}W(x,y)\,\mu(dy).
\end{equation}
The corresponding \(N\)-particle system is
\begin{equation}\label{eq:intro-particle}
dX_t^{i,N}
=
-\nabla V(X_t^{i,N})\,dt
-\frac1N\sum_{j=1}^N
\nabla_g^{(1)}W(X_t^{i,N},X_t^{j,N})\,dt
+\sqrt{2}\,dB_t^{i,M},
\qquad i=1,\ldots,N,
\end{equation}
where \(B^{1,M},\ldots,B^{N,M}\) are independent Brownian motions on \(M\).
The precise formulation of the Brownian motions and the stochastic
equations on \(M\) is given in Section~\ref{sec:preliminaries}, and the
assumptions are stated in Hypotheses~3.1 and~3.2. We impose regularity and
Lipschitz-type conditions on the interaction force, while the confinement
is required to be dissipative only at large distances rather than globally
geodesically convex.

Our geometric coupling construction follows the work of
Kendall~\cite{Kendall1986} and Cranston~\cite{Cranston1991} on reflection
coupling on Riemannian manifolds.  The 
distance formula \eqref{eq:distance-reflection-identity} entails both  the confinement and the geometry.
Through the assumptions in Hypothesis 3.1, the confinement drift tends to contract the distance between the
coupled processes, whereas negative curvature contributes an expansive
term through the second variation of distance. 
More precisely, the geometric contribution is expressed by a sum of index
forms along the minimizing geodesic joining the coupled particles.
For a unit-speed minimizing geodesic \(\gamma:[0,r]\to M\), write
\(\mathcal I(\gamma)\) for the corresponding endpoint-index term. If
$$
\mathcal I(\gamma)\leq H(r)
$$ for a function $H$
and the confinement contribution is bounded below by \(rC_V(r)\), then
geometry and confinement enter the radial distance estimate through the
combination
$$
D(r):=rC_V(r)-H(r)
$$
(see the two-point distance estimate \eqref{eq:f-rho}).
Hypothesis~\ref{hyp:geometry-confinement}\textup{(ii)} requires
$$
\liminf_{r\to\infty}\frac{D(r)}r>0.
$$
This large-distance
condition allows us to construct an increasing concave function \(f\)
satisfying

$$
4f''(r)-D(r)f'(r)\leq -C_f f(r)
$$
for some \(C_f>0\). Thus the concavity of the transportation cost compensates
for possible short-range expansion, while the confinement dominates the
geometric contribution at large distances.

A uniform Ricci lower bound
\(\operatorname{Ric}_g\geq-(n-1)k^2g\) is a convenient sufficient
condition, since it yields
$$
\mathcal I(\gamma)\leq(n-1)k^2r,
$$
but our argument only requires an upper bound on the index term itself.\\
\indent
Evolutionary McKean--Vlasov equations and their particle approximations on
manifolds have been studied before. In particular, Arnaudon and Del
Moral~\cite{ArnaudonDelMoral2019} obtained stability and uniform
propagation-of-chaos estimates using variational methods and synchronous
couplings. Their manifold propagation-of-chaos result, Theorem~4.3 of
\cite{ArnaudonDelMoral2019}, assumes a uniform Ricci lower bound in
addition to a dissipativity condition. Here we formulate the geometric
assumption directly in terms of the second variation of distance.

Appendix~\ref{app:curvature-spikes} shows that these assumptions genuinely
allow manifolds without a uniform Ricci lower bound. We construct a
complete Cartan--Hadamard surface with negative curvature concentrated
in thin annuli, and extend the example to higher dimensions by taking
products with Euclidean space. Its Ricci curvature is unbounded below,
but the endpoint-index term satisfies
\(\mathcal I(\gamma)\leq C\min\{r^{1/9},1\}\) and $D(r)= \alpha r-C\min\{r^{1/9}, 1\}$ for some $\alpha>0$.
Besides, a quadratic confinement potential
and smooth interactions built from compactly supported functions satisfy
the remaining assumptions when the interaction strength is sufficiently
small. 


We now state the main quantitative consequences. Under
Hypotheses~\ref{hyp:geometry-confinement} and~\ref{hyp:interaction}, the
coupling construction yields an increasing concave function \(f\),
equivalent to the geodesic distance. Let \(W_f\) denote the corresponding
Wasserstein distance. If \(\mu_t\) is the nonlinear law flow and the
particles are initially independent with common law
\(\nu_0\in\mathcal P_2(M)\), Theorem~\ref{thm:propagation-chaos} gives,
for some \(\lambda,C>0\),
\begin{equation}\label{eq:intro-chaos}
W_{\ell^1(f)}\left(
\mathcal L(X_t^{1,N},\ldots,X_t^{N,N}),\mu_t^{\otimes N}
\right)
\leq
e^{-\lambda t}W_f(\nu_0,\mu_0)+\frac{C}{\sqrt N},
\qquad t\geq0,
\end{equation}
where \(C\) is independent of \(t\) and \(N\). In particular, for a common
initial law this gives an \(N^{-1/2}\) propagation-of-chaos estimate
uniformly in time. Theorem~\ref{thm:convergence-equilibrium} also gives
the contraction estimate
$$
W_f(\mu_t,\nu_t)
\leq e^{-\lambda t}W_f(\mu_0,\nu_0),
$$
and Theorem~\ref{thm:invariant} yields a unique invariant law in
\(\mathcal P_2(M)\) and
exponential convergence to equilibrium. Since \(f\) is equivalent to the
geodesic distance, the corresponding conclusions also hold in \(W_1\).

The regularity assumptions on the interaction force are essential to our
coupling argument. Singular interactions require different methods.
Jabin and Wang~\cite{JabinWang2018} developed a relative-entropy approach
for singular kernels, including the two-dimensional Biot--Savart kernel,
and Guillin, Le Bris, and Monmarch\'e
~\cite{GuillinLeBrisMonmarche2021} used this approach to obtain
uniform-in-time propagation of chaos for the two-dimensional vortex model
and related systems. Rosenzweig and Serfaty
~\cite{RosenzweigSerfaty2023} obtained global-in-time mean-field
convergence for singular Riesz-type diffusive flows by a modulated-energy
method. Extending such methods to the present geometric setting is a
natural direction for future work.

Another possible direction is to investigate a manifold analogue of the
bi-coupling method of Ren and Wang~\cite{RenWang2023}, which yields
entropy estimates for McKean--Vlasov diffusions, and its application to
particle approximations on manifolds.

The paper is organized as follows.
Section~\ref{sec:preliminaries} introduces the stochastic and geometric
preliminaries. Section~\ref{sec:assumptions} states the hypotheses and
gives examples. Sections~\ref{sec:well-posedness}, \ref{sec:chaos}, and
\ref{sec:equilibrium} prove, respectively, global well-posedness with
uniform moment bounds, propagation of chaos, and contraction with
convergence to equilibrium. Appendix~\ref{app:mixed-hessian} contains
the mixed-Hessian estimate, and Appendix~\ref{app:curvature-spikes}
contains the curvature-spike construction and the verification of its
assumptions.

\section*{Acknowledgments}
M. Chae was supported by the National Research Foundation of Korea
(NRF) grant funded by the Korea government (No.~RS-2023-00279920), and
G. Yang was supported by NSFC grant No. 12601263.

\section{Preliminaries}\label{sec:preliminaries}

We use standard results on geodesics, index forms, and comparison geometry
from \cite{Petersen2016}, including the radial Laplacian estimate of
\cite{ChengYau1975}. Manifold Brownian motion and SDEs follow
\cite{Elworthy1982,Hsu2002,Emery1989}. For reflection coupling and its
modified-cost formulation, see
\cite{LindvallRogers1986,Kendall1986,Cranston1991,Eberle2016}.
We use the coupling formulation of Wasserstein distances from
\cite{Villani2009,AmbrosioGigliSavare2008}; background on propagation of
chaos is given in
\cite{McKean1967,Sznitman1991,ChaintronDiez2022,
DurmusEberleGuillinZimmer2020}.

Let \((M,g)\) be a complete connected \(n\)-dimensional Riemannian manifold with distance
\(\rho\). For \(y\notin\operatorname{Cut}(x)\cup\{x\}\), set
\(r=\rho(x,y)\) and let \(\gamma\) be the unique unit-speed minimizing
geodesic from \(x\) to \(y\). We write
\[
\gamma:[0,r]\to M,\quad \gamma(0)=x,\quad \gamma(r)=y,
\qquad
\gamma'(x):=\dot\gamma(0),\quad \gamma'(y):=\dot\gamma(r).
\]

\begin{proposition}[First variation formula; see \cite{Petersen2016}]
	\label{prop:first-variation}
	In the preceding notation,
	\[
	\nabla_x\rho(x,y)=-\gamma'(x),
	\qquad
	\nabla_y\rho(x,y)=\gamma'(y).
	\]
\end{proposition}

We first apply Proposition~\ref{prop:first-variation} where \(\rho\) is
smooth and then extend the resulting identities by localization and the
smooth coupling approximation introduced below; see
\cite{Hsu2002,Emery1989}.

\subsection*{Horizontal lifts and Brownian motion on a manifold}

For \(x\in M\), let \(O_x(M)\) be the set of linear isometries
\(u:\mathbb R^n\to T_xM\). The orthonormal frame bundle and its projection
are
\[
O(M):=\bigcup_{x\in M}O_x(M),
\qquad
\pi:O(M)\to M,\quad \pi(u)=x\ \text{for }u\in O_x(M).
\]

A frame \(u\in O_x(M)\) converts Euclidean noise into tangent noise
through the isometry \(v\mapsto uv\in T_xM\), allowing an
\(\mathbb R^n\)-valued Brownian motion to drive a diffusion on \(M\).

\begin{definition}[Horizontal curves and spaces;
see \cite{Elworthy1982,Hsu2002,Emery1989}]
	\label{def:horizontal}
	A smooth curve \(u_t\in O(M)\), with \(x_t=\pi(u_t)\), is horizontal if
	\[
	\nabla_{\dot x_t}(u_t e_i)=0,\qquad i=1,\ldots,n.
	\]
	The velocities of horizontal curves through \(u\) form the horizontal
	subspace \(\mathcal H_u\subset T_uO(M)\).
\end{definition}

The vertical space and the connection-induced splitting are
\[
\mathcal V_u:=\ker(d\pi_u),
\qquad
T_uO(M)=\mathcal H_u\oplus\mathcal V_u.
\]
Vertical directions change the frame at a fixed base point, whereas
horizontal directions move the base point while parallel transporting
the frame.

For the standard basis \(e_1,\ldots,e_n\) of \(\mathbb R^n\), let \(H_i\)
be the canonical horizontal vector field characterized by
\(H_i(u)\in\mathcal H_u\) and \(d\pi_u(H_i(u))=ue_i\). For
\(a=(a^1,\ldots,a^n)\), set
\[
H_a(u):=\sum_{i=1}^n a^iH_i(u),
\qquad
d\pi_u(H_a(u))=ua.
\]

\begin{definition}[Horizontal lift of a vector field; see \cite{Elworthy1982,Hsu2002}]
	\label{def:horizontal-lift}
	Let \(b\) be a vector field on \(M\). Its horizontal lift is the unique
	horizontal vector field \(\widetilde b\) on \(O(M)\) satisfying
	\[
	d\pi_u(\widetilde b(u))=b(\pi(u)).
	\]
\end{definition}

Let \(W=(W^1,\ldots,W^n)\) be a standard Brownian motion in
\(\mathbb R^n\), and let \(U_t\) solve
\begin{equation}\label{eq:horizontal-BM}
dU_t=\sqrt2\sum_{i=1}^n H_i(U_t)\circ dW_t^i.
\end{equation}
Then \(B_t^M:=\pi(U_t)\) is Brownian motion on \(M\) with generator
\(\Delta_g\); see \cite{Elworthy1982,Hsu2002,Emery1989}. Throughout,
\(dB_t^M\) denotes this projected horizontal noise, not a differential
taking values in a fixed Euclidean space.

For a smooth vector field \(b\) with horizontal lift \(\widetilde b\),
the diffusion with drift \(b\) is the projection \(X_t=\pi(U_t)\) of
\begin{equation}\label{eq:lifted-SDE-drift}
dU_t=\widetilde b(U_t)\,dt
+\sqrt2\sum_{i=1}^nH_i(U_t)\circ dW_t^i.
\end{equation}
In local charts of \(O(M)\), this is a Euclidean Stratonovich SDE with
locally Lipschitz coefficients. Standard local theory therefore gives a
unique maximal strong solution up to its lifetime; see
\cite{Emery1989,Elworthy1982,Hsu2002}.
We call \(X_t=\pi(U_t)\) the strong solution of
\(dX_t=b(X_t)\,dt+\sqrt2\,dB_t^M\). Its horizontal-development and
martingale formulations are equivalent; see
\cite[Proposition~3.2.1]{Hsu2002}.

\begin{proposition}[Generator formula;
see \cite{Elworthy1982,Hsu2002,Emery1989}]
	\label{prop:generator-formula}
	Let \(U_t\) solve \eqref{eq:lifted-SDE-drift} and \(X_t=\pi(U_t)\).
	For every \(\varphi\in C_c^\infty(M)\),
	\[
	\begin{aligned}
	L_b\varphi&=\Delta_g\varphi+\langle b,\nabla\varphi\rangle,\\
	d\varphi(X_t)&=(L_b\varphi)(X_t)\,dt
	+\sqrt2\sum_{i=1}^n(U_te_i)\varphi(X_t)\,dW_t^i.
	\end{aligned}
	\]
\end{proposition}

\subsection*{The McKean--Vlasov SDE and the particle system}

For \(\mu\in\mathcal P_1(M)\), define
\[
\mathbf F_\mu(x):=\int_M\nabla_g^{(1)}W(x,y)\,\mu(dy),
\qquad
b_\mu(x):=-\nabla V(x)-\mathbf F_\mu(x),
\]
where \(V\) and \(W\) satisfy
Hypotheses~\ref{hyp:geometry-confinement} and~\ref{hyp:interaction}.
The McKean--Vlasov SDE is
\begin{equation}\label{eq:McKean-Vlasov-SDE}
d\bar X_t=b_{\mu_t}(\bar X_t)\,dt+\sqrt2\,dB_t^M,
\qquad \mu_t=\mathcal L(\bar X_t).
\end{equation}

The corresponding \(N\)-particle system is
\begin{equation}\label{eq:particle-system}
\begin{aligned}
dX_t^{i,N}
&=-\nabla V(X_t^{i,N})\,dt
-\frac1N\sum_{j=1}^N\nabla_g^{(1)}W(X_t^{i,N},X_t^{j,N})\,dt\\
&\quad+\sqrt2\,dB_t^{i,M},\qquad i=1,\ldots,N,
\end{aligned}
\end{equation}
where \(B^{1,M},\ldots,B^{N,M}\) are independent Brownian motions on \(M\).

Given the law flow \(\mu_t=\mathcal L(\bar X_t)\), let
\(\bar X^1,\ldots,\bar X^N\) be independent nonlinear copies satisfying
\begin{equation}\label{eq:nonlinear-copies}
\begin{aligned}
d\bar X_t^i
&=-\nabla V(\bar X_t^i)\,dt-\mathbf F_{\mu_t}(\bar X_t^i)\,dt
+\sqrt2\,d\bar B_t^{i,M},\\
\mathcal L(\bar X_t^i)&=\mu_t,\qquad i=1,\ldots,N.
\end{aligned}
\end{equation}
In the proof of Theorem~\ref{thm:propagation-chaos}, each pair
\((\bar B^{i,M},B^{i,M})\) is joined by the mixed coupling of
Definition~\ref{def:mixed} to control
\(\rho(\bar X_t^i,X_t^{i,N})\).

\begin{proposition}[Weak Fokker--Planck equation;
see \cite{Hsu2002,AmbrosioGigliSavare2008}]
	\label{prop:FP-weak}
	If \(\bar X_t\) solves \eqref{eq:McKean-Vlasov-SDE} and
	\(\mu_t=\mathcal L(\bar X_t)\), then, for every
	\(\varphi\in C_c^\infty(M)\),
	\[
	\frac{d}{dt}\int_M\varphi\,d\mu_t
	=
	\int_M\!\left[\Delta_g\varphi
	-\langle\nabla V+\mathbf F_{\mu_t},\nabla\varphi\rangle\right]d\mu_t.
	\]
	Equivalently, in the distributional sense,
	\[
	\partial_t\mu_t
	=\Delta_g\mu_t+\nabla\!\cdot\!\bigl((\nabla V+\mathbf F_{\mu_t})\mu_t\bigr).
	\]
\end{proposition}

\subsection*{Reflection coupling and the distance process}

Let \(x\ne y\) with \(y\notin\operatorname{Cut}(x)\), let \(\gamma\)
be the minimizing geodesic from \(x\) to \(y\), and let
\(P_{x,y}:T_xM\to T_yM\) denote parallel transport along \(\gamma\).
Define
\[
R_{x,y}v
:=P_{x,y}v-2\langle v,\gamma'(x)\rangle\gamma'(y),
\qquad v\in T_xM.
\]
This isometry reverses the radial component and parallel transports the
orthogonal component. It is the intrinsic analogue of reflection across
the hyperplane orthogonal to \(y-x\).

\begin{definition}[Reflection coupling;
see \cite{LindvallRogers1986,Kendall1986,Cranston1991}]
	\label{def:reflection-coupling}
	Let \(X_t=\pi(U_t)\) and \(Y_t=\pi(V_t)\). Away from the diagonal and
	cut locus, set
	\[
	A_t:=V_t^{-1}R_{X_t,Y_t}U_t\in O(n),
	\qquad
	d\widetilde W_t=A_t\,dW_t.
	\]
	For drift fields \(b_X,b_Y\), the reflection coupling is defined by
	\[
	\begin{aligned}
	dU_t&=\widetilde b_X(U_t)\,dt
	+\sqrt2\sum_{i=1}^nH_i(U_t)\circ dW_t^i,\\
	dV_t&=\widetilde b_Y(V_t)\,dt
	+\sqrt2\sum_{i=1}^nH_i(V_t)\circ d\widetilde W_t^i.
	\end{aligned}
	\]
	Since \(A_t\) is orthogonal, Lévy's characterization shows that
	\(\widetilde W\) is Brownian motion. Thus the marginal generators are
	preserved while the radial noise is reflected.
\end{definition}

We next record the exact distance identity; the endpoint-index bound will
be inserted later. Let \(r_t=\rho(X_t,Y_t)\), let
\(\gamma_t:[0,r_t]\to M\) be the minimizing geodesic, and choose an
orthonormal basis \(E_1,\ldots,E_{n-1}\perp\gamma_t'(0)\). Let \(J_a\)
be the Jacobi field satisfying
\[
J_a(0)=E_a,\qquad J_a(r_t)=P_tE_a,
\]
where \(P_t\) is parallel transport along \(\gamma_t\). Define
\[
\begin{aligned}
I_{\gamma_t}(J,K)
&:=\int_0^{r_t}\!
\bigl(\langle\nabla_{\gamma_t'}J,\nabla_{\gamma_t'}K\rangle
-\langle R(J,\gamma_t')\gamma_t',K\rangle\bigr)\,ds,\\
\mathcal I_t&:=\sum_{a=1}^{n-1}I_{\gamma_t}(J_a,J_a).
\end{aligned}
\]

\begin{proposition}[Reflection-coupling distance formula]
	\label{prop:distance-reflection}
	See \cite{Kendall1986,Cranston1991} and
	\cite[Section~6.6]{Hsu2002}.
	Let \((X_t,Y_t)\) be a reflection coupling with drift fields \(b_X,b_Y\).
	Away from the diagonal and cut locus,
	\begin{align}
		dr_t
		={}&2\sqrt2\,d\beta_t+\mathcal I_t\,dt
		\nonumber\\
		&+\bigl(
		\langle b_Y(Y_t),\gamma_t'(r_t)\rangle
		-\langle b_X(X_t),\gamma_t'(0)\rangle
		\bigr)\,dt,
		\label{eq:distance-reflection-identity}
	\end{align}
	where \(\beta\) is a one-dimensional Brownian motion.
\end{proposition}

The factor \(2\sqrt2\) is the difference of the reflected radial noises
under the normalization with generator \(\Delta_g\). Transverse noise
produces \(\mathcal I_t\,dt\) through the second variation of distance,
while the two drift terms follow from
Proposition~\ref{prop:first-variation}.

\begin{proposition}[Ricci control of the endpoint-index term;
see \cite{Petersen2016}]
	\label{prop:index-ricci-control}
	If \(\operatorname{Ric}_g\ge -(n-1)k^2g\), then, away from the cut locus,
	\[
	\mathcal I_t
	=\sum_{a=1}^{n-1}I_{\gamma_t}(J_a,J_a)
	\le (n-1)k^2r_t.
	\]
\end{proposition}

For \(o\in M\), set \(\rho_o(x):=\rho(x,o)\). If
\(\operatorname{Ric}_g\ge -(n-1)k^2g\), then, away from
\(\{o\}\cup\operatorname{Cut}(o)\),
\begin{equation}\label{eq:laplacian-comparison}
\rho_o\,\Delta_g\rho_o\le (n-1)(1+k\rho_o);
\end{equation}
see \cite[Section~7.1.4]{Petersen2016}.

\begin{definition}[Synchronous coupling]\label{def:sync}
Let \(X_t=\pi(U_t)\) and \(Y_t=\pi(V_t)\). Away from the diagonal and
cut locus, synchronous coupling replaces reflection by parallel transport:
\[
A_t^{\mathrm{syn}}
:=V_t^{-1}P_{X_t,Y_t}U_t\in O(n),
\qquad
d\widetilde W_t=A_t^{\mathrm{syn}}\,dW_t.
\]
By Lévy's characterization, \(\widetilde W\) is again an
\(n\)-dimensional Brownian motion.
\end{definition}
\begin{proposition}[Synchronous distance formula]
\label{prop:sync-dist}
Let \((X_t,Y_t)\) be a synchronous coupling with drift fields \(b_X,b_Y\).
Away from the diagonal and cut locus,
\[
dr_t
=
\bigl(
\mathcal I_t+\langle b_Y(Y_t),\gamma_t'(r_t)\rangle
-\langle b_X(X_t),\gamma_t'(0)\rangle
\bigr)\,dt.
\]
If, in addition, \(\operatorname{Ric}_g\ge -(n-1)k^2g\), then
\[
dr_t
\le
\bigl(
(n-1)k^2r_t+\langle b_Y(Y_t),\gamma_t'(r_t)\rangle
-\langle b_X(X_t),\gamma_t'(0)\rangle
\bigr)\,dt.
\]
\end{proposition}
\begin{proof}
Set \(u_i:=U_te_i\). By the first-variation formula, the martingale part
of \(dr_t\) is
\[
\sqrt2\sum_{i=1}^n
\bigl(
-\langle\gamma_t'(0),u_i\rangle
+\langle\gamma_t'(r_t),P_{X_t,Y_t}u_i\rangle
\bigr)\,dW_t^i.
\]
Since \(P_{X_t,Y_t}\gamma_t'(0)=\gamma_t'(r_t)\), every summand vanishes.
Synchronous and reflection coupling agree in the transverse directions,
so their second-order contribution is the same endpoint-index term
\(\mathcal I_t\,dt\). Adding the drift contributions from the
first-variation formula proves the identity. The final inequality follows
from Proposition~\ref{prop:index-ricci-control}.
\end{proof}
\begin{definition}[Mixed coupling;
cf. \cite{DurmusEberleGuillinZimmer2020}]
\label{def:mixed}
	For \(\delta>0\), choose Lipschitz functions
	\(\phi_r^\delta,\phi_s^\delta:[0,\infty)\to[0,1]\) such that
	\[
	(\phi_r^\delta)^2+(\phi_s^\delta)^2=1,\qquad
	\phi_r^\delta=0\ \text{on }[0,\delta/2],\qquad
	\phi_r^\delta=1\ \text{on }[\delta,\infty).
	\]
	Let \(W,W'\) be independent \(\mathbb R^n\)-valued Brownian motions.
	Away from the cut locus, set
	\[
	r_t:=\rho(X_t,Y_t),\quad
	R_t:=V_t^{-1}R_{X_t,Y_t}U_t,\quad
	P_t:=V_t^{-1}P_{X_t,Y_t}U_t,
	\]
	and define \(R_t:=P_t\) when \(r_t\le\delta/2\). Drive the lifts by
	\[
	\begin{aligned}
	dW_t^X&=\phi_r^\delta(r_t)\,dW_t+\phi_s^\delta(r_t)\,dW_t',\\
	dW_t^Y&=\phi_r^\delta(r_t)R_t\,dW_t
	+\phi_s^\delta(r_t)P_t\,dW_t'.
	\end{aligned}
	\]
	Thus the first component is reflected and the second is transported
	synchronously.
\end{definition}

Because \(R_t,P_t\in O(n)\), \(W\) and \(W'\) are independent, and
\((\phi_r^\delta)^2+(\phi_s^\delta)^2=1\), the quadratic covariations of
both \(W^X\) and \(W^Y\) equal \(tI_n\). Lévy's characterization therefore
shows that both are \(n\)-dimensional Brownian motions. Hence, whenever
the marginal SDEs are unique in law, their laws are independent of the
choice of cutoff. The choices \(\phi_r^\delta\equiv1\) and
\(\phi_r^\delta\equiv0\) formally recover reflection and synchronous
coupling, respectively.

\begin{proposition}[Mixed-coupling distance formula]
\label{prop:mixed-dist}
	Let \((X_t,Y_t)\) be a mixed coupling with drift fields \(b_X,b_Y\), and
	let \(M^r\) denote the martingale part of \(r_t=\rho(X_t,Y_t)\). Away
	from the diagonal and cut locus,
	\begin{align}
		dr_t
		={}&2\sqrt2\,\phi_r^\delta(r_t)\,d\beta_t+\mathcal I_t\,dt
		\nonumber\\
		&+\bigl(
		\langle b_Y(Y_t),\gamma_t'(r_t)\rangle
		-\langle b_X(X_t),\gamma_t'(0)\rangle
		\bigr)\,dt,
		\label{eq:distance-mixed-identity}
	\end{align}
	and
	\[
	d\langle M^r\rangle_t
	=8\bigl(\phi_r^\delta(r_t)\bigr)^2\,dt,
	\]
	where \(\beta\) is a one-dimensional Brownian motion.
\end{proposition}

\begin{proof}
	For \(v\in T_{X_t}M\),
	\[
	R_{X_t,Y_t}v-P_{X_t,Y_t}v
	=
	-2\langle v,\gamma_t'(0)\rangle\gamma_t'(r_t),
	\]
	so reflection and parallel transport agree in every transverse direction.
	Their second-order contribution is therefore the same term
	\(\mathcal I_t\,dt\). In the radial direction, the synchronous
	\(W'\)-terms cancel, whereas the reflected \(W\)-terms have opposite
	signs. Hence the martingale part is
	\(2\sqrt2\,\phi_r^\delta(r_t)\,d\beta_t\), with quadratic variation
	\(8(\phi_r^\delta(r_t))^2dt\). The remaining drift terms follow from the
	first-variation formula.
\end{proof}

\begin{remark}\label{rem:mixed-diagonal}
Because \(\phi_r^\delta=0\) on \([0,\delta/2]\), the coupling is
synchronous near the diagonal and no reflection direction is required
there. At the cut locus, we use the standard localization and smooth
approximation for manifold couplings, apply the preceding formula on the
smooth locus, and pass to the limit; see \cite{Kendall1986,Hsu2002}.
The estimates below are uniform under this approximation.
\end{remark}

\subsection*{Modified transportation costs}

Let \(f:[0,\infty)\to[0,\infty)\) be increasing with \(f(0)=0\), and let
\(\Pi(\mu,\nu)\) denote the couplings of \(\mu,\nu\). Define
\[
W_f(\mu,\nu)
:=
\inf_{\pi\in\Pi(\mu,\nu)}
\int_{M\times M}f(\rho(x,y))\,\pi(dx,dy),
\]
For \(\mathbf x,\mathbf y\in M^N\), define
\[
\begin{aligned}
\ell_f^1(\mathbf x,\mathbf y)
&:=\frac1N\sum_{i=1}^Nf(\rho(x_i,y_i)),\\
W_{\ell^1(f)}(\mu,\nu)
&:=\inf_{\pi\in\Pi(\mu,\nu)}
\int_{M^N\times M^N}\ell_f^1(\mathbf x,\mathbf y)\,\pi(d\mathbf x,d\mathbf y).
\end{aligned}
\]

\begin{proposition}[Coordinate coupling bound]
		\label{prop:coordinate-coupling-bound}
		Let \(\mathbf X\) and \(\mathbf Y\) be \(M^N\)-valued random variables
		on a common probability space. Then
		\[
		W_{\ell^1(f)}(\mathcal L(\mathbf X),\mathcal L(\mathbf Y))
		\leq
		\mathbb E\,\ell_f^1(\mathbf X,\mathbf Y)
		=
		\frac1N\sum_{i=1}^N\mathbb E f\bigl(\rho(X^i,Y^i)\bigr).
		\]
\end{proposition}

\begin{proof}
	The joint law of \((\mathbf X,\mathbf Y)\) couples
	\(\mathcal L(\mathbf X)\) and \(\mathcal L(\mathbf Y)\); evaluating the
	infimum defining \(W_{\ell^1(f)}\) at this coupling gives the result.
\end{proof}

\section{Assumptions and examples}\label{sec:assumptions}

\begin{hypothesis}[Geometric and confinement conditions]
	\label{hyp:geometry-confinement}
	Assume \(V\in C^2(M)\). For distinct \(x,y\) off the cut locus, let
	\(\gamma:[0,r]\to M\) be the unit-speed minimizing geodesic, where
	\(r=\rho(x,y)\). We require:
	\begin{enumerate}
		\item There exists a continuous function
		\(C_V:[0,\infty)\to\mathbb R\) such that
		\[
		\langle\nabla V(y),\gamma'(r)\rangle
		-\langle\nabla V(x),\gamma'(0)\rangle
		\ge C_V(r)r,
		\qquad
		\liminf_{r\to\infty}C_V(r)> 0.
		\]
		\item Let \(E_1,\ldots,E_{n-1}\) be a parallel orthonormal frame normal
		to \(\gamma\), and let \(J_a\) be the Jacobi field satisfying
		\[
		J_a(0)=E_a(0),\qquad J_a(r)=E_a(r).
		\]
		There exists a continuous nondecreasing function
		\(H:[0,\infty)\to[0,\infty)\), with \(H(0)=0\), such that
		\[
		\mathcal I(\gamma)
		:=\sum_{a=1}^{n-1}I_\gamma(J_a,J_a)
		\le H(r).
		\]
		Define
		\begin{equation}
		\label{eq:endpoint-D}
		D(r):=rC_V(r)-H(r),
		\end{equation}
		and assume that \(D\) is continuous, \(D(0)=0\), and
		\begin{equation}
		\label{eq:endpoint-D-tail}
		\liminf_{r\to\infty}\frac{D(r)}r>0.
		\end{equation}

		\item Fix \(o\in M\) and set \(r_o(x):=\rho(o,x)\). Assume that
		\(L_0,L_1\ge0\) exist such that
		\begin{equation}
		\label{eq:radial-bound}
		r_o\Delta_g r_o\le L_0+L_1r_o.
		\end{equation}
		classically on \(M\setminus(\{o\}\cup\operatorname{Cut}(o))\) and in the
		barrier sense on \(\operatorname{Cut}(o)\).
	\end{enumerate}
\end{hypothesis}

\begin{hypothesis}[Interaction conditions]
\label{hyp:interaction}
	Let \(f,c_0,C_f\) be supplied by
	Lemma~\ref{lemma:cost-function}. Assume \(W\in C^2(M\times M)\) and that
	there exists \(\eta>0\) such that:
	\begin{enumerate}
		\item The self-interaction force vanishes:
		\[
		\nabla_g^{(1)}W(x,x)=0,\qquad x\in M.
		\]
		\item For distinct \(x,y\) off the cut locus, let
		\(\gamma:[0,r]\to M\) be the minimizing geodesic, with
		\(r=\rho(x,y)\). Then, for all \(z,w\in M\),
		\[
		\left|
		\langle\nabla_g^{(1)}W(x,z),\gamma'(0)\rangle
		-
		\langle\nabla_g^{(1)}W(y,w),\gamma'(r)\rangle
		\right|
		\le
		\eta\bigl[f(r)+f(\rho(z,w))\bigr].
		\]
		\item For all \(x,z,w\in M\),
		\[
		\left|
		\nabla_g^{(1)}W(x,z)-\nabla_g^{(1)}W(x,w)
		\right|_g
		\le
		\eta f(\rho(z,w)).
		\]	
%
	\end{enumerate}
\end{hypothesis}

\begin{Remark}\label{remC_V}
	\begin{enumerate}
		\item Let \(\ell_V:=\liminf_{r\to\infty}C_V(r)>0\). For every
		\(m_V\in(0,\ell_V)\), there exists \(M_V\ge0\) such that
		\begin{equation}\label{eq:coercivity-mV-MV}
	\langle\nabla V(y),\gamma'(r)\rangle
	-\langle\nabla V(x),\gamma'(0)\rangle
	\ge m_Vr-M_V,
	\qquad r=\rho(x,y).
	\end{equation}
	Indeed, choose \(R\) such that \(C_V(r)\ge m_V\) for \(r\ge R\), and set
	\[
	M_V:=\max_{0\le r\le R}r\bigl(m_V-C_V(r)\bigr)_+.
	\]

		\item The two quantitative restrictions have different roles:
		\(m_V>2\eta\) closes the uniform moment estimate in
		Section~\ref{sec:well-posedness}, whereas \(C_f>2\eta\) yields the strict
		contraction rate in Sections~\ref{sec:chaos}--\ref{sec:equilibrium}.
		Neither inequality is needed for local well-posedness, which follows from
		the local regularity of the coefficients.

		\item If \(\operatorname{Ric}_g\ge -(n-1)k^2g\), then one may take
		\[
		H(r)=(n-1)k^2r,\qquad L_0=n-1,\qquad L_1=(n-1)k.
		\]
			Consequently, \eqref{eq:endpoint-D-tail} becomes
			\(\liminf_{r\to\infty}C_V(r)>(n-1)k^2\). The converse fails: on the
			curvature-spike manifolds constructed in
			Appendix~\ref{app:curvature-spikes},
			Hypothesis~\ref{hyp:geometry-confinement}\textup{(ii)--(iii)} holds while
			\(\inf_M\operatorname{Ric}=-\infty\).
	\end{enumerate}
	\end{Remark}

\subsection{A basic Cartan--Hadamard example}

\begin{example}[Cartan--Hadamard manifolds]\label{ex:cartan-hadamard}
	Assume \(-k^2\le\operatorname{Sec}\le0\), fix \(o\in M\), and set
	\(r(x):=\rho(o,x)\). Let
	\[
	V(x)=r(x)^p-ar(x)^q,\qquad a>0,\quad p>q\ge2,
	\]
	and let
	\[
	W(x,y)=\Phi(\rho(x,y)),
	\]
	where \(\Phi\in C^\infty([0,\infty))\),
	\(\Phi'(0)=0\), and \(\Phi',\Phi''\in L^\infty\).
	Then the geometric and interaction hypotheses hold for the cost \(f\)
	constructed in Lemma~\ref{lemma:cost-function}. After replacing
	\(\Phi\) by \(\varepsilon\Phi\) with \(\varepsilon>0\) sufficiently small,
	one also has \(m_V>2\eta\) and \(C_f>2\eta\).
\end{example}

\begin{proof}[Verification]
	For \(x\ne o\), write \(r_o=\rho(o,x)\). The chain rule gives
	\[
	\begin{aligned}
	\operatorname{Hess}V
	&=\bigl[p(p-1)r_o^{p-2}-aq(q-1)r_o^{q-2}\bigr]\,
	dr_o\otimes dr_o\\
	&\quad+\bigl[pr_o^{p-1}-aqr_o^{q-1}\bigr]\,
	\operatorname{Hess}r_o.
	\end{aligned}
	\]
	Hessian comparison gives
	\[
	\operatorname{Hess}r_o
	\ge \frac1{r_o}\bigl(g-dr_o\otimes dr_o\bigr).
	\]
	For all sufficiently large \(r_o\),
	\(pr_o^{p-1}-aqr_o^{q-1}\ge0\), and hence
	\[
	\begin{aligned}
	\operatorname{Hess}V
	\ge{}&\bigl(pr_o^{p-2}-aqr_o^{q-2}\bigr)g\\
	&+\bigl(p(p-2)r_o^{p-2}-aq(q-2)r_o^{q-2}\bigr)
	dr_o\otimes dr_o.
	\end{aligned}
	\]
	Both coefficients tend to infinity. Choose \(\lambda_V>2(n-1)k^2\).
	Then \(\operatorname{Hess}V\ge\lambda_Vg\) outside some ball \(B(o,R)\).

	Set
	\[
	K_R:=
	\max\left\{0,\,
	\sup_{\substack{z\in\overline{B(o,R)}\\ |v|_z=1}}
	\bigl(-\operatorname{Hess}V_z(v,v)\bigr)\right\}.
	\]
	Since metric balls are geodesically convex, a minimizing geodesic
	\(\gamma:[0,d]\to M\) spends at most length \(2R\) in \(B(o,R)\).
	Therefore
	\[
	\begin{aligned}
	\langle\nabla V(y),\gamma'(d)\rangle
	-\langle\nabla V(x),\gamma'(0)\rangle
	&=\int_0^d\operatorname{Hess}V(\gamma',\gamma')\,ds\\
	&\ge \lambda_Vd-2R(\lambda_V+K_R).
	\end{aligned}
	\]
	The global bound \(\operatorname{Hess}V\ge-K_Rg\) also gives the lower
	bound \(-K_Rd\). Hence Hypothesis~\ref{hyp:geometry-confinement}\textup{(i)}
	holds with \(C_V(0):=-K_R\) and, for \(d>0\),
	\[
	C_V(d):=
	\max\left\{-K_R,\,
	\lambda_V-\frac{2R(\lambda_V+K_R)}d\right\}.
	\]
	In particular,
	\[
	\lim_{d\to\infty}C_V(d)=\lambda_V>(n-1)k^2.
	\]

	Set \(s:=\rho(x,y)\). Since \(\Phi'(0)=0\), the mean value theorem gives
	\begin{equation}
		\label{eq:example-phi-prime-over-r}
		\frac{|\Phi'(s)|}{s}\le\|\Phi''\|_\infty,\qquad s>0.
	\end{equation}
	Thus the apparent singularity at the diagonal is removable,
	\(W\in C^2(M\times M)\), and
	\(\nabla_g^{(1)}W(x,x)=0\).
	Away from the diagonal,
	\begin{align*}
		\operatorname{Hess}_{xx}W
		&=
		\Phi''(s)\,d_xs\otimes d_xs
		+
		\Phi'(s)\operatorname{Hess}_{xx}s,\\
		\operatorname{Hess}_{xy}W
		&=
		\Phi''(s)\,d_xs\otimes d_ys
		+
		\Phi'(s)\operatorname{Hess}_{xy}s.
	\end{align*}
	Hessian comparison and
	Lemma~\ref{lem:mixed-hessian-cartan} give
	\[
		|\operatorname{Hess}_{xx}s|
		\leq k\coth(ks),
		\qquad
		|\operatorname{Hess}_{xy}s|
		\leq s^{-1},
	\]
	with \(k\coth(ks)=s^{-1}\) when \(k=0\).
	Using \(\coth u\leq1+u^{-1}\) and
	\eqref{eq:example-phi-prime-over-r}, we obtain
	\begin{align*}
		|\operatorname{Hess}_{xx}W|
		&\leq
		2\|\Phi''\|_\infty+k\|\Phi'\|_\infty,\\
		|\operatorname{Hess}_{xy}W|
		&\leq
		2\|\Phi''\|_\infty.
	\end{align*}
	The remaining Hessian blocks satisfy the same bounds by symmetry. Let
	\(C_W\) bound all four blocks. Integrating along the minimizing geodesics
	from \(x\) to \(y\) and from \(z\) to \(w\) gives
	\[
	\begin{aligned}
	\bigl|
	\langle\nabla_1W(x,z),\gamma'(0)\rangle
	-\langle\nabla_1W(y,w),\gamma'(r)\rangle
	\bigr|
	&\le C_W\bigl(\rho(x,y)+\rho(z,w)\bigr),\\
	\bigl|\nabla_1W(x,z)-\nabla_1W(x,w)\bigr|_g
	&\le C_W\rho(z,w).
	\end{aligned}
	\]
	Since \(c_0u\le f(u)\),
	Hypothesis~\ref{hyp:interaction}\textup{(ii)--(iii)} holds with
	\(\eta=C_W/c_0\). Replacing \(\Phi\) by \(\varepsilon\Phi\) replaces
	\(\eta\) by \(\varepsilon\eta\).

	Since \(\operatorname{Sec}\ge-k^2\),
	\(\operatorname{Ric}_g\ge-(n-1)k^2g\), so
	\[
	H(r)=(n-1)k^2r,
	\qquad
	r_o\Delta_gr_o\le(n-1)+(n-1)kr_o.
	\]
	Moreover,
	\[
	\lim_{r\to\infty}\frac{D(r)}r
	=\lambda_V-(n-1)k^2>0.
	\]
	Thus Hypothesis~\ref{hyp:geometry-confinement} holds, and
	Lemma~\ref{lemma:cost-function} provides \(f,c_0,C_f\). Choose any
	\(m_V\in(0,\lambda_V)\). Since scaling \(\Phi\) scales \(\eta\) but does
	not change \(m_V\) or \(C_f\), taking \(\varepsilon>0\) sufficiently small
	ensures
	\[
	2\eta<\min\{m_V,C_f\}.
	\]
\end{proof}
\begin{lemma}[Construction of the cost function]
\label{lemma:cost-function}
	Assume Hypothesis~\ref{hyp:geometry-confinement}\textup{(i)--(ii)}.
	Then there exist \(C_f,c_0>0\) and an increasing concave function
	\(f\in C^2([0,\infty))\) such that, for \(r\ge0\),
	\begin{equation}\label{eq:endpoint-cost-ODE}
	\begin{aligned}
	&f(0)=0,\qquad c_0r\le f(r)\le r,\qquad 0<f'(r)\le1,\\
	&4f''(r)-D(r)f'(r)\le-C_ff(r).
	\end{aligned}
	\end{equation}
\end{lemma}

\begin{proof}
	By \eqref{eq:endpoint-D-tail}, choose \(d>0\) and \(R_0>0\) such that
	\(D(r)\ge dr\) for \(r\ge R_0\). Set \(D^-(r):=(-D(r))_+\) and define
	\[
	\varphi(r):=\exp\!\left(-\frac14\int_0^rD^-(s)\,ds\right),
	\qquad
	\Phi(r):=\int_0^r\varphi(s)\,ds.
	\]
	Then \(\varphi\) is constant on \([R_0,\infty)\); denote this value by
	\(\varphi_\infty>0\).

	Choose \(R_1>R_0\) and
	\(\chi\in C^\infty([0,\infty);[0,1])\) with \(\chi=1\) on
	\([0,R_0]\) and \(\chi=0\) on \([R_1,\infty)\). Define
	\[
	\begin{aligned}
	A&:=\int_0^{R_1}\chi(s)\frac{\Phi(s)}{\varphi(s)}\,ds>0,\\
	g(r)&:=1-\frac1{2A}\int_0^r\chi(s)\frac{\Phi(s)}{\varphi(s)}\,ds,
	\qquad
	f(r):=\int_0^r\varphi(s)g(s)\,ds.
	\end{aligned}
	\]
	By construction, \(\frac12\le g\le1\) and
	\(\varphi_\infty\le\varphi\le1\). Hence
	\[
	0<f'\le1,
	\qquad
	\frac{\varphi_\infty}{2}r\le f(r)\le r.
	\]
	Thus one may take \(c_0=\varphi_\infty/2\). Moreover,
	\[
	\varphi'=-\frac14D^-\varphi,
	\qquad
	g'=-\frac1{2A}\chi\frac{\Phi}{\varphi},
	\]
	and therefore \(f''=\varphi'g+\varphi g'\le0\). Writing
	\(D^+=\max\{D,0\}\), we have
	\[
	4f''-Df'=-D^+\varphi g+4\varphi g'.
	\]
	Since \(\chi=1\) and \(f\leq\Phi\) on \([0,R_0]\), while
	\(D(r)\geq dr\), \(\varphi\geq\varphi_\infty\),
	\(g\geq\tfrac12\), and \(f(r)\leq r\) on \([R_0,\infty)\),
	\[
	\begin{aligned}
	4f''(r)-D(r)f'(r)
	&\leq-\frac{2\Phi(r)}{A}
	\leq-\frac{2}{A}f(r),
	&&0\leq r\leq R_0,\\
	4f''(r)-D(r)f'(r)
	&\leq-\frac{d\varphi_\infty}{2}r
	\leq-\frac{d\varphi_\infty}{2}f(r),
	&&r\geq R_0.
	\end{aligned}
	\]
	Consequently, with
	\[
	C_f:=\min\left\{\frac{2}{A},\frac{d\varphi_\infty}{2}\right\}>0,
	\]
	we obtain
	\[
	4f''(r)-D(r)f'(r)\leq-C_f f(r),
	\qquad r\geq0.
	\]
	Finally, \(D(0)=\Phi(0)=0\) implies
	\(\varphi'(0)=g'(0)=0\), and hence \(f''(0)=0\).
\end{proof}

\begin{remark}[Size of the contraction rate]\label{rem:cf-size}
	The rates in Theorems~\ref{thm:propagation-chaos} and
	\ref{thm:convergence-equilibrium} depend on \(C_f\).
	With the notation of Lemma~\ref{lemma:cost-function}, set
	\begin{equation}
	\Lambda:=\int_0^{R_0}D^-(s)\,ds .
	\label{eq:cf-barrier}
	\end{equation}
	Since \(D^-=0\) on \([R_0,\infty)\),
	\(\varphi_\infty=e^{-\Lambda/4}\). Moreover,
	\[
	A\le\int_0^{R_1}\frac{\Phi(s)}{\varphi(s)}\,ds
	\le e^{\Lambda/4}\,\frac{R_1^2}{2},
	\]
	and therefore
	\begin{equation}
	C_f\ \ge\ e^{-\Lambda/4}
	\min\left\{\frac{4}{R_1^2},\ \frac d2\right\}.
	\label{eq:cf-lower}
	\end{equation}

	Since \(D^-(s)=\bigl(H(s)-sC_V(s)\bigr)^+\), if
	\(C_V\ge-\kappa_V\) on \([0,R_0]\), then
	\begin{equation}
	\Lambda\le\int_0^{R_0}H(s)\,ds+\frac{\kappa_VR_0^2}{2}.
	\label{eq:cf-barrier-bound}
	\end{equation}
	Consequently,
	\begin{equation}
	C_f\ \ge\
	\exp\left\{-\frac14\int_0^{R_0}H(s)\,ds-\frac{\kappa_VR_0^2}{8}\right\}
	\min\left\{\frac4{R_1^2},\ \frac d2\right\}.
	\label{eq:cf-lower-hyp}
	\end{equation}
	Thus curvature affects the contraction rate through the endpoint-index
	bound \(H\) on the region where the net dissipation \(D\) is not yet
	positive. Under the Ricci lower bound of Remark~\ref{remC_V}, this becomes
	\[
	C_f\ \ge\
	\exp\left\{-\frac{\bigl(\kappa_V+(n-1)k^2\bigr)R_0^2}{8}\right\}
	\min\left\{\frac4{R_1^2},\ \frac d2\right\}.
	\]
\end{remark}

The coupling used below is synchronous near the diagonal, so the
second-order term in \eqref{eq:endpoint-cost-ODE} is switched off there.
The following lemma quantifies the resulting defect.

\begin{lemma}[Near-diagonal defect]\label{lem:cost-defect}
	Let \(f\) and \(C_f\) be as in Lemma~\ref{lemma:cost-function}, let
	\(\delta>0\), and let \(\phi^\delta_r\) be as in
	Definition~\ref{def:mixed}.  Set
	\[
	\omega(\delta):=\sup_{0\le s\le\delta}D(s)^-,
	\qquad
	e(\delta):=\omega(\delta)+C_ff(\delta),
	\]
	where \(D^-:=\max(-D,0)\).  Then
	\begin{equation}
		-D(r)f'(r)+4f''(r)\phi^\delta_r(r)^2
		\le
		-C_ff(r)+e(\delta),
		\qquad r\ge0,
		\label{eq:cost-defect}
	\end{equation}
	and \(e(\delta)\to0\) as \(\delta\to0^+\).
\end{lemma}

\begin{proof}
	For \(r\geq\delta\), \(\phi_r^\delta(r)=1\), so the claim follows
	directly from \eqref{eq:endpoint-cost-ODE}. For \(0\leq r<\delta\),
	the concavity of \(f\) and \(0<f'\leq1\) give
	\[
	4f''(r)\phi_r^\delta(r)^2\leq0,
	\qquad
	-D(r)f'(r)\leq D^-(r)\leq\omega(\delta).
	\]
	Since \(f\) is increasing,
	\[
	\omega(\delta)
	\leq-C_ff(r)+C_ff(\delta)+\omega(\delta)
	=-C_ff(r)+e(\delta),
	\]
	which proves \eqref{eq:cost-defect}. Finally, continuity of \(D\) and
	\(f\), together with \(D(0)=f(0)=0\), yields
	\(e(\delta)\to0\) as \(\delta\downarrow0\).
\end{proof}

\section{Well-posedness}\label{sec:well-posedness}
We establish global well-posedness of
\eqref{eq:McKean-Vlasov-SDE}. Smoothness of \(V\) and \(W\) provides
local well-posedness for each frozen measure flow, while
Lemma~\ref{lem:interaction-field} controls the dependence of the drift
on that flow. Since \(M\) is noncompact, local existence alone does not
exclude explosion or provide the moments required in the empirical
interaction estimate \eqref{2nd}. We therefore combine a law-flow fixed
point argument with stopped distance-moment estimates that are uniform
in the stopping radius. These estimates yield non-explosion, uniform
moments, and continuation of the solution for all times.

\begin{lemma}\label{lem:interaction-field}
	Assume Hypothesis~\ref{hyp:interaction}. For
	\(\nu,\widetilde\nu\in\mathcal P_1(M)\) and \(x\in M\),
	\[
	\left|\mathbf F_\nu(x)-\mathbf F_{\widetilde\nu}(x)\right|
	\leq\eta\,\mathcal W_1(\nu,\widetilde\nu).
	\]
	Consequently, if
	\(\nu\in C([0,T];\mathcal P_1(M))\), then
	\(t\mapsto\mathbf F_{\nu_t}(x)\) is continuous for every \(x\in M\).
\end{lemma}

\begin{proof}
	For any \(\pi\in\Pi(\nu,\widetilde\nu)\),
	Hypothesis~\ref{hyp:interaction}\textup{(iii)} and \(f(r)\le r\) give
	\[
	\left|\mathbf F_\nu(x)-\mathbf F_{\widetilde\nu}(x)\right|
	\leq
	\eta\int f(\rho(z,w))\,\pi(dz,dw)
	\leq\eta\int\rho(z,w)\,\pi(dz,dw).
	\]
	Taking the infimum over \(\pi\) proves the estimate, and its final
	assertion follows immediately.
\end{proof}
\begin{prop}[Global well-posedness and uniform moment bound]
	\label{prop:MV-moment-control}
	Let \(o\in M\), \(p\geq2\), and suppose that
	Hypotheses~\ref{hyp:geometry-confinement} and
	\ref{hyp:interaction} hold, with \(m_V>2\eta\), and that
	\[
	\mathbb E\,\rho(o,\bar X_0)^p<\infty.
	\]
	For \(\epsilon\in(0,p(m_V-2\eta))\), let \(C_\epsilon\) be the
	constant in Lemma~\ref{lem:frozen-moment} and set
	\begin{equation}
	M:=\max\left\{
	\mathbb E\,\rho(o,\bar X_0)^p,\,
	\frac{C_\epsilon}{p(m_V-2\eta)-\epsilon}
	\right\}.
	\label{eq:moment-ball-radius}
	\end{equation}
	Then \eqref{eq:McKean-Vlasov-SDE} has a unique non-explosive strong
	solution on \([0,\infty)\). Its law flow belongs to
	\(C([0,\infty);\mathcal P_p(M))\), is unique in
	\(C([0,\infty);\mathcal P_1(M))\), and satisfies
	\begin{equation}
	\sup_{t\geq0}\int_{M}\rho(o,z)^p\,\mu_t(dz)\leq M.
	\label{eq:MV-uniform-moment}
	\end{equation}
\end{prop}
To prove the proposition, fix
\(\nu\in C([0,\infty);\mathcal P_1(M))\) and consider the frozen SDE
\begin{equation}\label{frozen}
dX_t^\nu
=b_{\nu_t}(X_t^\nu)\,dt+\sqrt2\,dB_t^M,
\qquad
X_0^\nu=\bar X_0.
\end{equation}

\begin{lemma}[Frozen moment estimate]\label{lem:frozen-moment}
	Let \(p\geq2\) and let
	\(\nu\in C([0,\infty);\mathcal P_1(M))\) satisfy
	\[
	\bar m:=\sup_{t\geq0}
	\left(\int_{M}\rho(o,z)^p\,\nu_t(dz)\right)^{1/p}<\infty .
	\]
	Assume Hypotheses~\ref{hyp:geometry-confinement} and
	\ref{hyp:interaction}, \(m_V>2\eta\), and
	\(\mathbb E\rho(o,X_0^\nu)^p<\infty\). For
	\(\epsilon\in(0,p(m_V-2\eta))\), set
	\[
	c_\epsilon:=p(m_V-2\eta)+\eta-\epsilon>0.
	\]
	Then the maximal local strong solution of \eqref{frozen} is
	non-explosive, and there exists a constant \(C_\epsilon\), independent
	of \(\nu\), such that
	\begin{equation}
	\mathbb E\rho(o,X_t^\nu)^p
	\leq e^{-c_\epsilon t}\mathbb E\rho(o,X_0^\nu)^p
	+\frac{\eta\bar m^p+C_\epsilon}{c_\epsilon}
	\bigl(1-e^{-c_\epsilon t}\bigr),
	\qquad t\geq0.
	\label{eq:frozen-moment-bound}
	\end{equation}
\end{lemma}

\begin{proof}
	Set \(r_t:=\rho(o,X_t^\nu)\), and let \(\xi\) be the maximal lifetime
	of \(X^\nu\). By the radial It\^o comparison and the first-variation
	formula, for \(t<\xi\),
	\[
	dr_t
	\leq \sqrt2\,dB_t+
	\left[
	\left\langle
	\mathbf F_{\nu_t}(X_t^\nu)+\nabla V(X_t^\nu),
	\gamma_{o,t}'(X_t^\nu)
	\right\rangle
	+\Delta_g r_o(X_t^\nu)
	\right]dt,
	\]
	where \(B\) is a one-dimensional Brownian motion and
	\(\gamma_{o,t}\) is the minimizing geodesic from \(X_t^\nu\) to \(o\).
	Applying It\^o's formula to \(r_t^p\) gives
	\begin{align}
	dr_t^p
	\leq{}&
	p\sqrt2\,r_t^{p-1}\,dB_t
	+p r_t^{p-1}
	\left\langle
	\mathbf F_{\nu_t}(X_t^\nu)+\nabla V(X_t^\nu),
	\gamma_{o,t}'(X_t^\nu)
	\right\rangle dt \nonumber\\
	&+p r_t^{p-1}\Delta_g r_o(X_t^\nu)\,dt
	+p(p-1)r_t^{p-2}\,dt.
	\label{eq:rho-p-moment}
	\end{align}
	Applying \eqref{eq:coercivity-mV-MV} along
	\(\gamma_{o,t}\) gives
	\begin{equation}
	\left\langle\nabla V(X_t^\nu),
	\gamma_{o,t}'(X_t^\nu)\right\rangle
	\leq-m_Vr_t+|\nabla V(o)|+M_V.
	\label{eq:estimate-1-Lp}
	\end{equation}
	By Hypothesis~\ref{hyp:interaction}\textup{(i)--(ii)},
	\(f\leq\rho\), and H\"older's inequality,
	\begin{align}
	\left|
	\left\langle\mathbf F_{\nu_t}(X_t^\nu),
	\gamma_{o,t}'(X_t^\nu)\right\rangle
	\right|
	&\leq\eta\left(r_t+\int_M\rho(o,z)\,\nu_t(dz)\right)\nonumber\\
	&\leq\eta(r_t+\bar m).
	\label{eq:estimate-2-Lp}
	\end{align}
	Moreover, Hypothesis~\ref{hyp:geometry-confinement}\textup{(iii)}
	yields
	\[
	r_t\Delta_g r_o(X_t^\nu)\leq L_0+L_1r_t.
	\]
	Hence, with \(A_o:=|\nabla V(o)|+M_V+L_1\),
	\begin{align*}
	dr_t^p
	\leq{}&
	\sqrt2p r_t^{p-1}\,dB_t\\
	&+\Bigl[
	-p(m_V-\eta)r_t^p
	+p\eta\bar m r_t^{p-1} \\
	&\hspace{1.5cm}
	+pA_or_t^{p-1}
	+p(L_0+p-1)r_t^{p-2}
	\Bigr]dt.
	\end{align*}
	For every \(\epsilon>0\), Young's inequality gives
	\[
	p\eta\bar m\,r^{p-1}\leq\eta(p-1)r^p+\eta\bar m^{\,p},
	\qquad
	pA_or^{p-1}+p(L_0+p-1)r^{p-2}
	\leq\epsilon r^p+C_\epsilon.
	\]
	where \(C_\epsilon\) is independent of \(\nu\). Hence, with
	\[
	c:=p(m_V-\eta)-\eta(p-1)-\epsilon
	=p(m_V-2\eta)+\eta-\epsilon=c_\epsilon
	\]
	and \(K:=\eta\bar m^p+C_\epsilon\), we obtain
	\[
	dr_t^p
	\leq\sqrt2p r_t^{p-1}\,dB_t
	+\bigl(-cr_t^p+K\bigr)dt.
	\]
	Multiplying by \(e^{ct}\) yields
	\[
	d\bigl(e^{ct}r_t^p\bigr)
	\leq\sqrt2p e^{ct}r_t^{p-1}\,dB_t+Ke^{ct}\,dt.
	\]

	For \(R>0\), define
	\[
	\tau_R:=\inf\{t\in[0,\xi):r_t\geq R\}.
	\]
	
	On \([0,t\wedge\tau_R]\), the stochastic integrand is bounded and
	hence has zero expectation. Therefore,
	\begin{equation}
	\mathbb E\!\left[e^{c(t\wedge\tau_R)}
	r_{t\wedge\tau_R}^p\right]
	\leq
	\mathbb E r_0^p+\frac Kc\bigl(e^{ct}-1\bigr).
	\label{eq:frozen-stopped}
	\end{equation}
	On \(\{\tau_R\leq t\}\), one has
	\(r_{t\wedge\tau_R}\geq R\), and thus
	\[
	\mathbb P(\tau_R\leq t)
	\leq
	\frac{1}{R^p}
	\left[
	\mathbb E r_0^p+\frac Kc\bigl(e^{ct}-1\bigr)
	\right]
	\longrightarrow0.
	\]
	Since a finite-time explosion on the complete manifold \(M\) forces
	the path to leave every compact metric ball,
	\[
	\{\xi\leq T\}
	\subseteq\bigcap_{R>0}\{\tau_R\leq T\},
	\]
	and hence \(\mathbb P(\xi\leq T)=0\) for every \(T>0\).
	Thus \(\xi=\infty\) almost surely. Letting \(R\to\infty\) in
	\eqref{eq:frozen-stopped} and applying Fatou's lemma now gives
	\[
	\mathbb E r_t^p
	\leq e^{-ct}\mathbb E r_0^p
	+\frac Kc\bigl(1-e^{-ct}\bigr),
	\]
	which is \eqref{eq:frozen-moment-bound}.
\end{proof}

\begin{proof}[Proof of Proposition~\ref{prop:MV-moment-control}]
	\emph{Step 1: an invariant complete space.}
	For \(T>0\), let
	\[
	\mathcal S_{T,M}:=
	\left\{\nu\in C([0,T];\mathcal P_1(M)):
	\nu_0=\mu_0,\ 
	\sup_{t\leq T}\int_{M}\rho(o,z)^p\,\nu_t(dz)\leq M
	\right\},
	\]
	equipped with
	\[
	d_T(\nu,\widetilde\nu):=
	\sup_{t\leq T}\mathcal W_1(\nu_t,\widetilde\nu_t).
	\]
	Since \((\mathcal P_1(M),\mathcal W_1)\) is complete and the
	\(p\)-moment is lower semicontinuous under weak convergence,
	\((\mathcal S_{T,M},d_T)\) is a nonempty complete metric space.

	For \(\nu\in\mathcal S_{T,M}\), extend \(\nu\) by
	\(\nu_t:=\nu_T\) for \(t>T\), solve \eqref{frozen}, and define
	\[
	\Gamma(\nu)_t:=\mathcal L(X_t^\nu),\qquad 0\leq t\leq T.
	\]
	Lemma~\ref{lem:frozen-moment} and \(\bar m^p\leq M\) give
	\[
	\sup_{t\leq T}\mathbb E\,\rho(o,X_t^\nu)^p
	\leq
	\max\left\{
	\mathbb E\rho(o,\bar X_0)^p,\,
	\frac{\eta M+C_\epsilon}{c_\epsilon}
	\right\}
	\leq M,
	\]
	where the last inequality follows from
	\[
	C_\epsilon
	\leq\bigl(p(m_V-2\eta)-\epsilon\bigr)M
	=(c_\epsilon-\eta)M.
	\]
	Moreover,
	\begin{equation}
	\mathcal W_1\bigl(\Gamma(\nu)_t,\Gamma(\nu)_s\bigr)
	\leq\mathbb E\rho(X_t^\nu,X_s^\nu).
	\label{eq:W1-by-coupling}
	\end{equation}
	Path continuity gives convergence in probability as \(t\to s\), while
	the uniform \(p\)-moment bound, with \(p>1\), gives uniform
	integrability. Hence the right-hand side tends to zero, and therefore
	\(\Gamma(\nu)\in\mathcal S_{T,M}\).

	\emph{Step 2: contraction.}
	Let \(\nu,\widetilde\nu\in\mathcal S_{T,M}\), and couple
	\(X:=X^\nu\) and \(Y:=X^{\widetilde\nu}\), starting from the same
	initial variable, by the mixed coupling of
	Definition~\ref{def:mixed}. Set \(r_t:=\rho(X_t,Y_t)\), so that
	\(r_0=0\). Proposition~\ref{prop:mixed-dist} gives
	\[
	dr_t\leq
	2\sqrt2\,\phi_r^\delta(r_t)\,d\beta_t+
	\left[
	H(r_t)+
	\langle b_{\widetilde\nu_t}(Y_t),\gamma_t'(Y_t)\rangle-
	\langle b_{\nu_t}(X_t),\gamma_t'(X_t)\rangle
	\right]dt.
	\]
	Insert and subtract
	\(\mathbf F_{\nu_t}(Y_t)\) in the drift difference.
	Hypothesis~\ref{hyp:geometry-confinement}\textup{(i)},
	Hypothesis~\ref{hyp:interaction}\textup{(ii)}, and
	Lemma~\ref{lem:interaction-field} then yield
	\[
	\begin{aligned}
	&\langle b_{\widetilde\nu_t}(Y_t),\gamma_t'(Y_t)\rangle-
	\langle b_{\nu_t}(X_t),\gamma_t'(X_t)\rangle\\
	&\qquad\leq-r_tC_V(r_t)+\eta f(r_t)
	+\eta\mathcal W_1(\nu_t,\widetilde\nu_t).
	\end{aligned}
	\]
	Recalling \(D(r)=rC_V(r)-H(r)\), we obtain
	\begin{equation}
	dr_t\leq
	2\sqrt2\,\phi_r^\delta(r_t)\,d\beta_t+
	\left[-D(r_t)+\eta f(r_t)
	+\eta\mathcal W_1(\nu_t,\widetilde\nu_t)\right]dt,
	\qquad r_0=0.
	\label{eq:contraction-D}
	\end{equation}
	Since \(f'\) is absolutely continuous, the It\^o--Tanaka formula,
	\eqref{eq:contraction-D}, and
	\(d\langle r\rangle_t=8\phi_r^\delta(r_t)^2dt\) give
	\[
	\begin{aligned}
	df(r_t)\leq{}&
	2\sqrt2\,\phi_r^\delta(r_t)f'(r_t)\,d\beta_t\\
	&+\Bigl[
	-D(r_t)f'(r_t)+4f''(r_t)\phi_r^\delta(r_t)^2
	+\eta f(r_t)f'(r_t)
	+\eta f'(r_t)\mathcal W_1(\nu_t,\widetilde\nu_t)
	\Bigr]dt.
	\end{aligned}
	\]
	Lemma~\ref{lem:cost-defect} and \(0<f'\leq1\) therefore yield
	\begin{equation}
	df(r_t)\leq
	2\sqrt2\,\phi_r^\delta(r_t)f'(r_t)\,d\beta_t+
	\left[-(C_f-\eta)f(r_t)
	+\eta\mathcal W_1(\nu_t,\widetilde\nu_t)+e(\delta)\right]dt.
	\label{eq:contraction-f}
	\end{equation}
	The stochastic integrand is bounded, so taking expectations and applying
	Gr\"onwall's inequality with \(a:=(\eta-C_f)_+\) gives
	\[
	\mathbb Ef(r_t)
	\le
	\eta\int_0^t e^{a(t-s)}
	\mathcal W_1(\nu_s,\widetilde\nu_s)\,ds
	+e(\delta)t e^{at},
	\]
	The law of each marginal of the mixed coupling is independent of
	\(\delta\). Hence
	\[
	W_f\bigl(\Gamma(\nu)_t,\Gamma(\widetilde\nu)_t\bigr)
	\leq\mathbb Ef(r_t).
	\]
	Letting \(\delta\downarrow0\) and using
	Lemma~\ref{lem:cost-defect}, we obtain
	\begin{equation}
	W_f\bigl(\Gamma(\nu)_t,\Gamma(\widetilde\nu)_t\bigr)
	\leq
	\eta\int_0^t e^{a(t-s)}
	\mathcal W_1(\nu_s,\widetilde\nu_s)\,ds.
	\label{eq:Wf-contraction}
	\end{equation}
	Since \(c_0r\leq f(r)\),
	\begin{equation}
	d_T\bigl(\Gamma(\nu),\Gamma(\widetilde\nu)\bigr)
	\leq\frac{\eta T e^{aT}}{c_0}\,
	d_T(\nu,\widetilde\nu).
	\label{eq:fixed-point-contraction}
	\end{equation}
	Choose \(T_*>0\), depending only on \(\eta,C_f,c_0\), such that
	\(\eta T_*e^{aT_*}/c_0\leq\tfrac12\). Then \(\Gamma\) is a contraction
	on \(\mathcal S_{T_*,M}\), so it has a unique fixed point
	\(\mu\). The corresponding frozen solution satisfies
	\(\mu_t=\mathcal L(X_t^\mu)\) and hence solves
	\eqref{eq:McKean-Vlasov-SDE} on \([0,T_*]\).

	\emph{Step 3: iteration.}
	Since \(\mu\in\mathcal S_{T_*,M}\),
	\[
	\int_{M}\rho(o,z)^p\,\mu_{T_*}(dz)\leq M.
	\]
	The construction can therefore be restarted at \(T_*\) with the same
	moment radius \(M\) and the same time step \(T_*\). Iteration yields a
	global non-explosive strong solution satisfying
	\eqref{eq:MV-uniform-moment}.

	\emph{Step 4: uniqueness.}
	Let \(\bar X^1,\bar X^2\) be two solutions with the same initial law,
	and suppose that their law flows \(\mu,\widetilde\mu\) belong to
	\(C([0,\infty);\mathcal P_1(M))\).
	Repeating the coupling estimate of Step 2 on \([0,T_*]\) gives
	\[
	d_{T_*}(\mu,\widetilde\mu)
	\leq\frac{\eta T_*e^{aT_*}}{c_0}
	d_{T_*}(\mu,\widetilde\mu)
	\leq\frac12d_{T_*}(\mu,\widetilde\mu).
	\]
	Hence \(\mu=\widetilde\mu\) on \([0,T_*]\), and iteration gives equality
	on \([0,\infty)\). If the two solutions have the same initial variable
	and are driven by the same Brownian motion, they solve the same frozen
	SDE with the common law flow. Pathwise uniqueness for that equation
	then gives \(\bar X^1=\bar X^2\) almost surely.
\end{proof}

\begin{corollary}[Particle system]\label{cor:particle-moment}
	Assume Hypotheses~\ref{hyp:geometry-confinement} and
	\ref{hyp:interaction}, with \(m_V>2\eta\). For \(N\geq1\), suppose
	\[
	\frac1N\sum_{i=1}^N
	\mathbb E\rho(o,X_0^{i,N})^p<\infty,
	\qquad p\geq2,
	\]
	and set \(\rho_i(t):=\rho(o,X_t^{i,N})\).
	Then \eqref{eq:particle-system} has a unique non-explosive strong
	solution and, for every
	\(\epsilon\in(0,p(m_V-2\eta))\),
	\[
	\sup_{t\geq0}\ \frac1N\sum_{i=1}^N\mathbb E\,\rho_i(t)^p
	\leq
	\max\left\{
	\frac1N\sum_{i=1}^N\mathbb E\rho_i(0)^p,\,
	\frac{C_\epsilon}{p(m_V-2\eta)-\epsilon}
	\right\},
	\]
	where \(C_\epsilon\) is independent of \(N\).
\end{corollary}

\begin{proof}
	Since \(V,W\in C^2\), the coefficients of the finite-dimensional SDE
	on \(M^N\) are locally Lipschitz; hence a unique strong solution
	exists up to its maximal lifetime. Set
	\[
	A_o:=|\nabla V(o)|+M_V+L_1,
	\qquad
	Y_t:=\frac1N\sum_{i=1}^N\rho_i(t)^p.
	\]
	The coordinatewise calculation from Lemma~\ref{lem:frozen-moment}
	gives
	\[
	\begin{aligned}
	d\rho_i^p
	\leq{}&
	\sqrt2p\rho_i^{p-1}\,d\beta_t^i\\
	&
	+p\rho_i^{p-1}\left[
	-m_V\rho_i+A_o+\frac\eta N\sum_{j=1}^N\bigl(\rho_i+\rho_j\bigr)
	\right]dt
	+p(L_0+p-1)\rho_i^{p-2}\,dt.
	\end{aligned}
	\]
	Moreover, by Young's inequality,
	\[
	\frac1{N^2}\sum_{i,j=1}^N\rho_i^{p-1}\rho_j
	\leq
	\frac1N\sum_{i=1}^N\rho_i^p=Y_t.
	\]
	Thus the interaction contributes at most \(2p\eta Y_t\). Applying the
	remaining Young inequalities before averaging gives
	\[
	dY_t\leq d\mathcal M_t+
	\left[-c_\epsilon Y_t+C_\epsilon\right]dt,
	\qquad
	c_\epsilon:=p(m_V-2\eta)-\epsilon>0,
	\]
	where \(\mathcal M\) is a local martingale and \(C_\epsilon\) is
	independent of \(N\).
	Multiplying by \(e^{c_\epsilon t}\) gives
	\[
	d\bigl(e^{c_\epsilon t}Y_t\bigr)
	\leq e^{c_\epsilon t}d\mathcal M_t
	+C_\epsilon e^{c_\epsilon t}dt.
	\]
	Let \(\xi_N\) be the maximal lifetime and define
	\[
	\tau_\ell:=
	\inf\left\{t\in[0,\xi_N):
	\max_{1\leq i\leq N}\rho_i(t)\geq\ell\right\}.
	\]
	The stopped stochastic integrand is bounded, so
	\[
	\mathbb E\!\left[
	e^{c_\epsilon(t\wedge\tau_\ell)}
	Y_{t\wedge\tau_\ell}\right]
	\leq
	\mathbb E\,Y_0+\frac{C_\epsilon}{c_\epsilon}
	\bigl(e^{c_\epsilon t}-1\bigr).
	\]
	On \(\{\tau_\ell\leq t\}\),
	\(Y_{t\wedge\tau_\ell}\geq\ell^p/N\); hence
	\[
	\mathbb P(\tau_\ell\leq t)
	\leq
	\frac{N}{\ell^p}
	\left[
	\mathbb EY_0+
	\frac{C_\epsilon}{c_\epsilon}
	\bigl(e^{c_\epsilon t}-1\bigr)
	\right]
	\longrightarrow0.
	\]
	Completeness of \(M^N\) therefore implies
	\(\xi_N=\infty\) almost surely. Letting \(\ell\to\infty\) and applying
	Fatou's lemma gives
	\[
	\mathbb EY_t
	\leq e^{-c_\epsilon t}\mathbb EY_0
	+\frac{C_\epsilon}{c_\epsilon}\bigl(1-e^{-c_\epsilon t}\bigr),
	\]
	which proves the claim.
\end{proof}

\section{Propagation of chaos}\label{sec:chaos}
We are now ready to prove the following propagation of chaos.

\begin{theorem}\label{thm:propagation-chaos}
	Assume Hypotheses~\ref{hyp:geometry-confinement}
	and~\ref{hyp:interaction}, with \(m_V>2\eta\) and \(C_f>2\eta\).
	Let \(\mu_0,\nu_0\in\mathcal P_2(M)\), let \(\mu_t\) be the law of the
	McKean--Vlasov solution starting from \(\mu_0\), and set
	\[
	C_2:=\sup_{t\ge0}
	\left(\mathbb E\rho(o,\bar X_t^i)^2\right)^{1/2}<\infty,
	\qquad
	\lambda:=C_f-2\eta>0.
	\]
	Suppose \(X_0^{1,N},\ldots,X_0^{N,N}\) are i.i.d. with law \(\nu_0\).
	For \(\xi\in\Pi(\mu_0,\nu_0)\), choose the initial pairs
	\((\bar X_0^i,X_0^{i,N})\) independently with common law \(\xi\).
	Then the mixed coupling with cutoff \(\delta>0\) satisfies
	\begin{align}
	\frac1N\sum_{i=1}^N
	\mathbb Ef\bigl(\rho(\bar X_t^i,X_t^{i,N})\bigr)
	\leq{}&
	e^{-\lambda t}
	\frac1N\sum_{i=1}^N
	\mathbb Ef\bigl(\rho(\bar X_0^i,X_0^{i,N})\bigr)\nonumber\\
	&+\frac{1-e^{-\lambda t}}{\lambda}
	\left(\frac{4\sqrt2\,\eta C_2}{\sqrt N}+e(\delta)\right).
	\label{eq:chaos-coordinate-estimate}
	\end{align}
	Consequently, letting \(\delta\downarrow0\) and optimizing over \(\xi\),
	\begin{equation}
	W_f\bigl(\mu_t^{N,\nu_0},\mu_t\bigr)
	\leq e^{-\lambda t}W_f(\nu_0,\mu_0)
	+\frac{4\sqrt2\,\eta C_2}{\lambda\sqrt N},
	\label{eq:chaos-one-particle}
	\end{equation}
	where \(\mu_t^{N,\nu_0}:=\mathcal L(X_t^{i,N})\). Moreover,
	\begin{equation}
	W_{\ell^1(f)}
	\left(
	\mathcal L(X_t^{1,N},\ldots,X_t^{N,N}),
	\mu_t^{\otimes N}
	\right)
	\leq e^{-\lambda t}W_f(\nu_0,\mu_0)
	+\frac{4\sqrt2\,\eta C_2}{\lambda\sqrt N}.
	\label{eq:chaos-wasserstein-estimate}
	\end{equation}
\end{theorem}
\begin{remark}\label{rem:cf-balance}
	The rate \(\lambda=C_f-2\eta\) expresses the competition among
	confinement, geometry, and interaction. Indeed,
	\[
	D(r)=rC_V(r)-H(r),
	\]
	so stronger confinement increases the net dissipation, whereas a larger
	endpoint-index contribution \(H\), reflecting stronger geodesic
	separation, decreases it. The cost construction converts \(D\) into
	\(C_f\), and the interaction produces the loss \(2\eta\).
	Remark~\ref{rem:cf-size} gives a quantitative lower bound for \(C_f\).
\end{remark}
\begin{proof}
	Proposition~\ref{prop:MV-moment-control} with $p=2$ and
	Corollary~\ref{cor:particle-moment} give global well-posedness and the
	uniform bound defining $C_2$.
	Fix \(\xi\in\Pi(\mu_0,\nu_0)\) and construct the initial pairs with
	joint law \(\xi^{\otimes N}\).  In particular, the nonlinear initial
	variables are i.i.d. with law \(\mu_0\), the particle initial variables
	are i.i.d. with law \(\nu_0\), and
	\[
	\frac1N\sum_{i=1}^N\mathbb E f\bigl(\rho(\bar X_0^i,X_0^{i,N})\bigr)
	=
	\int_{M\times M}f(\rho(x,y))\,\xi(dx,dy).
	\]
	Fix \(\delta>0\) and couple each pair by the mixed coupling of
	Definition~\ref{def:mixed}, and set
	\[
	\rho_t^i:=\rho(\bar X_t^i,X_t^{i,N}).
	\]
	Let \(\gamma_t^i\) be the minimizing geodesic from
	\(\bar X_t^i\) to \(X_t^{i,N}\), and define
	\[
	\begin{aligned}
	A_t^i:={}&
	\left\langle
	\frac1N\sum_{j=1}^N
	\nabla_g^{(1)}W(X_t^{i,N},X_t^{j,N}),
	(\gamma_t^i)'(X_t^{i,N})
	\right\rangle\\
	&-
	\left\langle
	\mathbf F_{\mu_t}(\bar X_t^i),
	(\gamma_t^i)'(\bar X_t^i)
	\right\rangle.
	\end{aligned}
	\]
	Proposition~\ref{prop:mixed-dist} and
	Hypothesis~\ref{hyp:geometry-confinement}\textup{(i)--(ii)} give
	\begin{align}
	df(\rho_t^i)\leq{}&
	2\sqrt2\,\phi_r^\delta(\rho_t^i)f'(\rho_t^i)\,d\beta_t^i\nonumber\\
	&+\left[
	4f''(\rho_t^i)\phi_r^\delta(\rho_t^i)^2
	-D(\rho_t^i)f'(\rho_t^i)-f'(\rho_t^i)A_t^i
	\right]dt.
	\label{eq:f-rho}
	\end{align}
	By Lemma~\ref{lem:cost-defect},
	\begin{equation}
	4f''(\rho_t^i)\phi_r^\delta(\rho_t^i)^2
	-D(\rho_t^i)f'(\rho_t^i)
	\leq-C_ff(\rho_t^i)+e(\delta).
	\label{eq:ODE-estimate}
	\end{equation}
	Decompose
	\[
	A_t^i=I_{1,t}^i+I_{2,t}^i,
	\]
	where
	\[
	\begin{aligned}
	I_{1,t}^i
	:=\frac1N\sum_{j=1}^N\Bigl[
	&\left\langle
	\nabla_g^{(1)}W(X_t^{i,N},X_t^{j,N}),
	(\gamma_t^i)'(X_t^{i,N})
	\right\rangle\\
	&-
	\left\langle
	\nabla_g^{(1)}W(\bar X_t^i,\bar X_t^j),
	(\gamma_t^i)'(\bar X_t^i)
	\right\rangle
	\Bigr]
	\end{aligned}
	\]
	and
	\[
	I_{2,t}^i
	:=
	\left\langle
	\frac1N\sum_{j=1}^N
	\nabla_g^{(1)}W(\bar X_t^i,\bar X_t^j)
	-\mathbf F_{\mu_t}(\bar X_t^i),
	(\gamma_t^i)'(\bar X_t^i)
	\right\rangle.
	\]
	For \(j\neq i\), set
	\[
	Z_j:=\nabla_g^{(1)}W(\bar X_t^i,\bar X_t^j),
	\qquad
	m_i:=\mathbf F_{\mu_t}(\bar X_t^i).
	\]
	Conditionally on \(\bar X_t^i\), the variables \(Z_j\), \(j\neq i\),
	are i.i.d. with mean \(m_i\). Hence
	\[
	\mathbb E\!\left[
	\left|m_i-\frac1{N-1}\sum_{j\neq i}Z_j\right|^2
	\,\middle|\,\bar X_t^i\right]
	=
	\frac1{N-1}
	\mathbb E\!\left[|Z_j-m_i|^2\mid\bar X_t^i\right].
	\]
	Using an independent pair \(x,y\) with law
	\(\mu_t\otimes\mu_t\), Hypothesis~\ref{hyp:interaction}\textup{(iii)}
	gives
	\begin{align}
	\mathbb E\left|
	m_i-\frac1{N-1}\sum_{j\neq i}Z_j
	\right|^2
	&\leq
	\frac{\eta^2}{N-1}
	\int_{M\times M}\rho(x,y)^2\,\mu_t(dx)\mu_t(dy)\nonumber\\
	&\leq\frac{4\eta^2C_2^2}{N-1}.
	\label{2nd}
	\end{align}
	Since \(Z_i=0\),
	\[
	\left|\frac1N\sum_{j=1}^NZ_j-m_i\right|
	\leq
	\left|\frac1{N-1}\sum_{j\neq i}Z_j-m_i\right|
	+\frac1{N(N-1)}\left|\sum_{j\neq i}Z_j\right|.
	\]
	Moreover, Hypothesis~\ref{hyp:interaction}\textup{(i),(iii)} gives
	\(\mathbb E|Z_j|\leq2\eta C_2\). Therefore,
	\begin{equation}
	\mathbb E|I_{2,t}^i|
	\leq
	\frac{2\eta C_2}{\sqrt{N-1}}
	+\frac{2\eta C_2}{N}
	\leq\frac{4\sqrt2\,\eta C_2}{\sqrt N}.
	\label{eq:I-2}
	\end{equation}
	Hypothesis~\ref{hyp:interaction}\textup{(ii)} gives
	\begin{equation}
	|I_{1,t}^i|
	\leq
	\eta f(\rho_t^i)
	+\frac{\eta}{N}\sum_{j=1}^N f(\rho_t^j).
	\label{eq:I-1}
	\end{equation}
	Set
	\[
	u_N(t):=\frac1N\sum_{i=1}^N\mathbb Ef(\rho_t^i).
	\]
	Combining \eqref{eq:f-rho}, \eqref{eq:ODE-estimate},
	\eqref{eq:I-1}, and \eqref{eq:I-2}, using \(f'\leq1\), and averaging
	over \(i\), we obtain
	\[
	u_N'(t)
	\leq-\lambda u_N(t)
	+\frac{4\sqrt2\,\eta C_2}{\sqrt N}+e(\delta).
	\]
	Gr\"onwall's inequality gives \eqref{eq:chaos-coordinate-estimate}.

	For each \(i\), the law of
	\((X_t^{i,N},\bar X_t^i)\) couples
	\(\mu_t^{N,\nu_0}\) and \(\mu_t\), while the joint law of the two
	configuration vectors couples
	\[
	\mathcal L(X_t^{1,N},\ldots,X_t^{N,N})
	\quad\text{and}\quad
	\mu_t^{\otimes N}.
	\]
	Proposition~\ref{prop:coordinate-coupling-bound} therefore bounds both
	transportation costs by \(u_N(t)\). Since the marginal laws are
	independent of \(\delta\), letting \(\delta\downarrow0\) and then taking
	the infimum over \(\xi\in\Pi(\mu_0,\nu_0)\) proves
	\eqref{eq:chaos-one-particle} and
	\eqref{eq:chaos-wasserstein-estimate}.
\end{proof}
	
\section{Convergence to equilibrium}\label{sec:equilibrium}

We now prove exponential contraction between two McKean--Vlasov law
flows.

\begin{theorem}[Exponential contraction]\label{thm:convergence-equilibrium}
	Assume Hypotheses~\ref{hyp:geometry-confinement} and
	\ref{hyp:interaction}, with \(m_V>2\eta\) and \(C_f>2\eta\), and set
	\[
	\lambda:=C_f-2\eta>0.
	\]
	Let \(\mu_t\) and \(\nu_t\) be the law flows of the McKean--Vlasov
	solutions starting from \(\mu_0,\nu_0\in\mathcal P_2(M)\).
	For every \(\pi\in\Pi(\mu_0,\nu_0)\) and \(\delta>0\), there exists a
	mixed coupling \((X_t^\delta,Y_t^\delta)\), with initial law \(\pi\)
	and marginals \(\mu_t,\nu_t\), such that
	\begin{equation}
	\mathbb Ef\bigl(\rho(X_t^\delta,Y_t^\delta)\bigr)
	\leq e^{-\lambda t}
	\mathbb Ef\bigl(\rho(X_0^\delta,Y_0^\delta)\bigr)
	+\frac{e(\delta)}{\lambda}\bigl(1-e^{-\lambda t}\bigr).
	\label{eq:equilibrium-coordinate-contraction}
	\end{equation}
	Consequently,
	\begin{equation}
	W_f(\mu_t,\nu_t)
	\leq e^{-\lambda t}W_f(\mu_0,\nu_0).
	\label{eq:equilibrium-Wf-contraction}
	\end{equation}
\end{theorem}

\begin{proof}
	Fix \(\pi\in\Pi(\mu_0,\nu_0)\) and \(\delta>0\), and let
	\((X_t,Y_t)\) be the corresponding mixed coupling. Set
	\[
	r_t:=\rho(X_t,Y_t).
	\]
	Proposition~\ref{prop:mixed-dist},
	Hypothesis~\ref{hyp:geometry-confinement}\textup{(i)--(ii)}, and
	Lemma~\ref{lem:cost-defect} control the curvature and confinement
	terms by
	\[
	4f''(r_t)\phi_r^\delta(r_t)^2-D(r_t)f'(r_t)
	\leq-C_ff(r_t)+e(\delta).
	\]

	Let \((X_t',Y_t')\) be an independent copy of \((X_t,Y_t)\).
	Conditioning on \((X_t,Y_t)\) and applying
	Hypothesis~\ref{hyp:interaction}\textup{(ii)} gives
	\[
	\begin{aligned}
	&\left|
	\left\langle\mathbf F_{\nu_t}(Y_t),\gamma_t'(Y_t)\right\rangle
	-\left\langle\mathbf F_{\mu_t}(X_t),\gamma_t'(X_t)\right\rangle
	\right|\\
	&\qquad\leq
	\eta f(r_t)+\eta\,\mathbb Ef(r_t).
	\end{aligned}
	\]
	Consequently, the It\^o--Tanaka formula and \(f'\leq1\) yield
	\[
	df(r_t)
	\leq d\mathcal M_t+
	\left[-(C_f-\eta)f(r_t)
	+\eta\,\mathbb Ef(r_t)+e(\delta)\right]dt,
	\]
	where \(\mathcal M\) is a martingale. Thus
	\[
	\frac{d}{dt}\mathbb Ef(r_t)
	\leq-\lambda\mathbb Ef(r_t)+e(\delta).
	\]
	Gr\"onwall's inequality proves
	\eqref{eq:equilibrium-coordinate-contraction}. Since the marginal laws
	are independent of \(\delta\), letting \(\delta\downarrow0\) and then
	taking the infimum over \(\pi\in\Pi(\mu_0,\nu_0)\) proves
	\eqref{eq:equilibrium-Wf-contraction}.
\end{proof}

For \(\mu\in\mathcal P_2(M)\), let
\[
P_t^*\mu:=\mathcal L(\bar X_t^\mu),
\]
where \(\bar X^\mu\) is the McKean--Vlasov solution with initial law
\(\mu\). Uniqueness implies the nonlinear semigroup property
\[
P_{t+s}^*=P_t^*P_s^*.
\]

\begin{theorem}[Invariant measure and convergence]\label{thm:invariant}
	Assume Hypotheses~\ref{hyp:geometry-confinement} and
	\ref{hyp:interaction}, with \(m_V>2\eta\) and \(C_f>2\eta\), and set
	\[
	\lambda:=C_f-2\eta>0.
	\]
	Then \((P_t^*)_{t\geq0}\) admits a unique invariant measure
	\(\mu_\infty\in\mathcal P_2(M)\). Moreover, for every
	\(\nu_0\in\mathcal P_2(M)\),
	\begin{equation}
	W_f(P_t^*\nu_0,\mu_\infty)
	\leq e^{-\lambda t}W_f(\nu_0,\mu_\infty),
	\qquad t\geq0.
	\label{eq:exp-conv}
	\end{equation}
\end{theorem}

\begin{proof}
	Fix \(o\in M\) and let \(\delta_o\) be the Dirac measure at \(o\).
	By Proposition~\ref{prop:MV-moment-control},
	\[
	C_o:=\sup_{s\geq0}W_f(P_s^*\delta_o,\delta_o)
	\leq\sup_{s\geq0}\mathbb E\rho(o,\bar X_s^{\delta_o})<\infty.
	\]
	Using the semigroup property and
	\eqref{eq:equilibrium-Wf-contraction}, for \(s,t\geq0\),
	\[
	\begin{aligned}
	W_f(P_{t+s}^*\delta_o,P_t^*\delta_o)
	&=W_f(P_t^*P_s^*\delta_o,P_t^*\delta_o)\\
	&\leq e^{-\lambda t}W_f(P_s^*\delta_o,\delta_o)
	\leq C_oe^{-\lambda t}.
	\end{aligned}
	\]
	Hence \((P_t^*\delta_o)_{t\geq0}\) is Cauchy in \(W_f\).
	Since \(c_0\rho\leq f(\rho)\leq\rho\), the metrics \(W_f\) and
	\(\mathcal W_1\) are equivalent, and completeness yields a probability
	measure \(\mu_\infty\) such that
	\[
	P_t^*\delta_o\longrightarrow\mu_\infty
	\quad\text{in }W_f.
	\]
	The uniform second-moment estimate and lower semicontinuity imply
	\(\mu_\infty\in\mathcal P_2(M)\).

	For fixed \(s\geq0\), the contraction estimate and semigroup property
	give
	\[
	\begin{aligned}
	W_f(P_s^*\mu_\infty,\mu_\infty)
	\leq{}&
	W_f(P_s^*\mu_\infty,P_s^*P_t^*\delta_o)
	+W_f(P_{t+s}^*\delta_o,\mu_\infty)\\
	\leq{}&
	e^{-\lambda s}W_f(\mu_\infty,P_t^*\delta_o)
	+W_f(P_{t+s}^*\delta_o,\mu_\infty).
	\end{aligned}
	\]
	Letting \(t\to\infty\) shows that
	\(P_s^*\mu_\infty=\mu_\infty\).

	For any \(\nu_0\in\mathcal P_2(M)\), invariance and
	\eqref{eq:equilibrium-Wf-contraction} give
	\[
	W_f(P_t^*\nu_0,\mu_\infty)
	=W_f(P_t^*\nu_0,P_t^*\mu_\infty)
	\leq e^{-\lambda t}W_f(\nu_0,\mu_\infty),
	\]
	which proves \eqref{eq:exp-conv}. If \(\widetilde\mu_\infty\) is another
	invariant measure in \(\mathcal P_2(M)\), then
	\[
	W_f(\mu_\infty,\widetilde\mu_\infty)
	\leq e^{-\lambda t}
	W_f(\mu_\infty,\widetilde\mu_\infty)
	\]
	for every \(t>0\), and hence the two measures coincide.
\end{proof}

\appendix

\section{A mixed Hessian estimate}
\label{app:mixed-hessian}

\begin{lemma}\label{lem:mixed-hessian-cartan}
	Let \(M\) be a Cartan--Hadamard manifold. For \(x\neq y\), set
	\(r:=\rho(x,y)\). Then
	\[
	\left|\operatorname{Hess}_{xy}\rho(X,Y)\right|
	\leq\frac{|X|\,|Y|}{r},
	\qquad
	X\in T_xM,\quad Y\in T_yM.
	\]
\end{lemma}

\begin{proof}
	Let \(\gamma:[0,r]\to M\) be the minimizing geodesic from \(x\) to
	\(y\), and let \(X^\perp,Y^\perp\) be the components orthogonal to
	\(\dot\gamma\) at the endpoints. The radial components do not
	contribute. Let \(J_X,J_Y\) be the normal Jacobi fields satisfying
	\[
	J_X(0)=X^\perp,\quad J_X(r)=0,
	\qquad
	J_Y(0)=0,\quad J_Y(r)=Y^\perp.
	\]
	The mixed second-variation formula gives
	\[
	\operatorname{Hess}_{xy}\rho(X,Y)
	=I_\gamma(J_X,J_Y)
	=\langle D_tJ_X(r),Y^\perp\rangle.
	\]

	Let \(A\) be the backward Jacobi tensor with
	\(A(r)=0\) and \(D_tA(r)=\operatorname{Id}\). Then
	\[
	J_X(t)=A(t)A(0)^{-1}X^\perp.
	\]
	Since \(\operatorname{Sec}\leq0\), Rauch comparison yields
	\[
	|A(t)v|\geq(r-t)|v|,
	\qquad 0\leq t\leq r,
	\]
	and hence \(\|A(0)^{-1}\|\leq r^{-1}\). Therefore,
	\[
	\left|\operatorname{Hess}_{xy}\rho(X,Y)\right|
	=
	\left|\left\langle A(0)^{-1}X^\perp,Y^\perp\right\rangle\right|
	\leq\frac{|X|\,|Y|}{r}.
	\]
\end{proof}

\section{A curvature-spike example}\label{app:curvature-spikes}

We construct a model satisfying the endpoint-index hypothesis although
its Ricci curvature has no finite lower bound, and then verify the
assumptions used in Sections~\ref{sec:well-posedness}--\ref{sec:equilibrium}.

\subsection{The metric and the endpoint-index bound}

Choose \(\psi\in C_c^\infty((-1,1))\) with
\(0\leq\psi\leq1\) and \(\psi(0)=1\), and define
\[
R_j:=2^j,\qquad A_j:=2^{4j},\qquad
\delta_j:=2^{-10j},\qquad
q(r):=\sum_{j\geq1}A_j
\psi\left(\frac{r-R_j}{\delta_j}\right).
\]
Let \(a\) solve
\[
a''=qa,\qquad a(0)=0,\qquad a'(0)=1,
\]
and equip \(\mathbb R^2\) with
\[
g_\Sigma=dr^2+a(r)^2d\theta^2.
\]
Because \(q\) vanishes near the origin, \(a(r)=r\) there, so the metric
extends smoothly across the pole.

\begin{proposition}[Curvature spikes and geodesic occupation]
\label{prop:spike-class}
The surface \((\Sigma,g_\Sigma)\) is complete and Cartan--Hadamard, and
\[
\mathcal K_\Sigma(r)=-q(r)\leq0,
\qquad
\inf_{\substack{x\in\Sigma\\ |v|_x=1}}
\operatorname{Ric}_x(v,v)=-\infty.
\]
There exist \(C_a,C_1<\infty\) such that
\[
r\leq a(r)\leq C_ar,\qquad 1\leq a'(r)\leq C_1.
\]
Moreover, every unit-speed geodesic segment \(\gamma\) of length \(s\)
satisfies
\begin{equation}
\int_\gamma q(r(\gamma(u)))\,du
\le C\min\{s^{1/9},1\}.
\label{eq:spike-occupation}
\end{equation}
\end{proposition}

\begin{proof}
Since \(a''=qa\geq0\), we have \(a'\geq1\) and \(a\geq r\). The
Volterra equation gives
\[
\frac{a(r)}r
\leq1+\int_0^r u q(u)\frac{a(u)}u\,du,
\qquad
\int_0^\infty u q(u)\,du
\leq3\sum_{j\geq1}R_jA_j\delta_j<\infty.
\]
Gr\"onwall's inequality yields \(a(r)\leq C_ar\), and
\[
a'(r)=1+\int_0^r q(u)a(u)\,du
\leq1+C_a\int_0^\infty u q(u)\,du=:C_1.
\]
The radial metric coefficient is one, so every curve escaping to
infinity has infinite length; hence \(\Sigma\) is complete. Moreover,
\(\mathcal K_\Sigma=-a''/a=-q\leq0\), and
\(q(R_j)=A_j\to\infty\). Thus \(\Sigma\) is Cartan--Hadamard and its
Ricci curvature has no finite lower bound.

It remains to prove \eqref{eq:spike-occupation}. Along a complete
unit-speed geodesic, Clairaut's relation gives
\[
\ell=a(r)^2\dot\theta,\qquad
\dot r^2=1-\frac{\ell^2}{a(r)^2}.
\]
The radial coordinate has at most one turning point \(r_*\), and each
radial shell is therefore crossed on at most two monotone branches.
It is determined by \(a(r_*)=|\ell|\). On either branch,
\[
\frac{du}{dr}
=
\frac{a(r)}
{\sqrt{(a(r)-a(r_*))(a(r)+a(r_*))}}
\leq C\sqrt{\frac r{r-r_*}},
\]
because \(a(r)-a(r_*)\geq r-r_*\),
\(a(r)+a(r_*)\geq r\), and \(a(r)\leq C_ar\).
Therefore, for \(S_j=[R_j-\delta_j,R_j+\delta_j]\),
\[
\operatorname{length}(\gamma\cap S_j)
\leq C\sqrt{R_j\delta_j},
\qquad
\int_{\gamma\cap S_j}q\,du
\leq CA_j\sqrt{R_j\delta_j}=C2^{-j/2}.
\]
The trivial bound \(\int_{\gamma\cap S_j}q\,du\leq A_js\) also gives
\[
\int_\gamma q\,du
\leq C\sum_{j\geq1}\min\{2^{4j}s,2^{-j/2}\}.
\]
For \(0<s\leq1\), choose \(J\) so that
\(2^{9J/2}\leq s^{-1}<2^{9(J+1)/2}\). Then
\[
\sum_{j\leq J}2^{4j}s\leq Cs^{1/9},
\qquad
\sum_{j>J}2^{-j/2}\leq Cs^{1/9}.
\]
For \(s\geq1\), the sum is uniformly bounded. This proves
\eqref{eq:spike-occupation}.
\end{proof}

\begin{corollary}\label{cor:spike-index}
For every minimizing geodesic \(\gamma\) of length \(r\) on \(\Sigma\),
\[
\mathcal I(\gamma)\le H_\Sigma(r),
\qquad
H_\Sigma(r):=C\min\{r^{1/9},1\}.
\]
For every \(n\geq2\), the same bound holds on
\(M_n:=\Sigma\times\mathbb R^{n-2}\), although its Ricci curvature has
no finite lower bound.
\end{corollary}

\begin{proof}
On \(\Sigma\), let \(E\) be the parallel unit normal field along
\(\gamma\). The index lemma and \(\mathcal K_\Sigma=-q\) give
\[
\mathcal I(\gamma)
=I_\gamma(J,J)
\leq I_\gamma(E,E)
=\int_\gamma q\,du
\leq H_\Sigma(r).
\]

For a unit-speed geodesic in \(M_n\), let \(c\in[0,1]\) be the constant
speed of its \(\Sigma\)-component. If \(c>0\), write that component as
\(\sigma(cs)\), where \(\sigma\) is unit speed. Using parallel comparison
fields in the product and then Proposition~\ref{prop:spike-class},
\[
\begin{aligned}
\mathcal I(\gamma)
&\leq-\int_0^r\operatorname{Ric}(\dot\gamma,\dot\gamma)\,ds\\
&=c^2\int_0^r q(\sigma(cs))\,ds\\
&=c\int_0^{cr}q(\sigma(u))\,du
\leq cH_\Sigma(cr)\leq H_\Sigma(r).
\end{aligned}
\]
The case \(c=0\) is immediate.
\end{proof}

\subsection{Compatible confinement and interaction potentials}

Let \(o\) be the pole of \(M_n\) and set
\[
V(x):=\frac{\alpha}{2}\rho(o,x)^2,
\qquad \alpha>0.
\]
Since \(M_n\) is Cartan--Hadamard, Hessian comparison gives
\[
\operatorname{Hess}
\left(\frac12\rho(o,\cdot)^2\right)\geq g.
\]
Thus Hypothesis~\ref{hyp:geometry-confinement}\textup{(i)} holds with
\[
C_V\equiv\alpha,\qquad m_V=\alpha,\qquad M_V=0.
\]
Together with Corollary~\ref{cor:spike-index}, this gives
\[
D(r)=\alpha r-C\min\{r^{1/9},1\},
\qquad
\lim_{r\to\infty}\frac{D(r)}r=\alpha>0.
\]
Let
\[
W_\varepsilon(x,y)
=\frac{\varepsilon}{2}\{h(x)-h(y)\}^2,
\qquad h\in C_c^\infty(M_n).
\]
Write
\[
B_0:=\|h\|_\infty,\qquad
B_1:=\|\nabla h\|_\infty,\qquad
B_2:=\|\operatorname{Hess}h\|_\infty,
\qquad
L_h:=B_1^2+2B_0B_2.
\]
Then
\[
\nabla_g^{(1)}W_\varepsilon(x,y)
=\varepsilon\{h(x)-h(y)\}\nabla h(x),
\qquad
\nabla_g^{(1)}W_\varepsilon(x,x)=0.
\]
Differentiation in the first variable and the mean-value estimate in
the second give
\[
\left|\nabla_x\nabla_g^{(1)}W_\varepsilon(x,y)\right|
\leq|\varepsilon|L_h,
\qquad
\left|
\nabla_g^{(1)}W_\varepsilon(x,z)
-\nabla_g^{(1)}W_\varepsilon(x,w)
\right|
\leq|\varepsilon|B_1^2\rho(z,w).
\]
Consequently, along the unique minimizing geodesic from \(x\) to \(y\),
\[
\left|
\nabla_g^{(1)}W_\varepsilon(x,z)
-P_{y\to x}\nabla_g^{(1)}W_\varepsilon(y,w)
\right|
\leq|\varepsilon|L_h
\bigl(\rho(x,y)+\rho(z,w)\bigr).
\]
Since \(c_0\rho\leq f(\rho)\), Hypothesis~\ref{hyp:interaction}
holds with
\[
\eta_\varepsilon:=\frac{|\varepsilon|L_h}{c_0}.
\]

It remains to verify the one-point radial condition in
Hypothesis~\ref{hyp:geometry-confinement}\textup{(iii)}.
Write \(o=(o_\Sigma,0)\), \(x=(x_\Sigma,z)\), and set
\[
r_\Sigma:=\rho_\Sigma(o_\Sigma,x_\Sigma),
\qquad
R:=\rho_{M_n}(o,x),
\qquad
Y(x):=R^2=r_\Sigma^2+|z|^2.
\]
Using Proposition~\ref{prop:spike-class},
\[
\Delta_{M_n}Y
=
2+2r_\Sigma\frac{a'(r_\Sigma)}{a(r_\Sigma)}+2(n-2)
\leq2+2C_1+2(n-2)=:C_\Delta,
\]
because \(a(r_\Sigma)\geq r_\Sigma\) and
\(a'(r_\Sigma)\leq C_1\). Since \(Y=R^2\) and
\(|\nabla R|=1\) away from \(o\),
\[
R\Delta_{M_n}R
=\frac{\Delta_{M_n}Y-2}{2}
\leq\frac{C_\Delta-2}{2}.
\]
Thus Hypothesis~\ref{hyp:geometry-confinement}\textup{(iii)} holds with
\[
L_0:=\frac{C_\Delta-2}{2},
\qquad
L_1:=0.
\]

Since
\[
\eta_\varepsilon=\frac{|\varepsilon|L_h}{c_0}\longrightarrow0
\qquad\text{as }\varepsilon\to0,
\]
choose
\[
|\varepsilon|
<\frac{c_0}{2L_h}\min\{\alpha,C_f\}.
\]
Then \(m_V=\alpha>2\eta_\varepsilon\) and
\(C_f>2\eta_\varepsilon\). Hence the model satisfies all the assumptions
of Proposition~\ref{prop:MV-moment-control},
Corollary~\ref{cor:particle-moment},
Theorem~\ref{thm:propagation-chaos}, and the results of
Section~\ref{sec:equilibrium}, despite having no finite Ricci lower
bound.

		\bibliographystyle{plain}
	\bibliography{manuscript_arxiv_20260911}
\end{document}